\documentclass[a4paper,12pt]{amsart} 
\usepackage{amssymb,euscript,verbatim,ulem}
\usepackage[utf8]{inputenc}      
\usepackage[T1]{fontenc}

\usepackage[polish,english]{babel}
\usepackage{tikz, hyperref, aliascnt, color, mathtools}
\usepackage{enumerate}

\usetikzlibrary{arrows,shapes,positioning}
\usetikzlibrary{decorations.markings}
\tikzstyle arrowstyle=[scale=1]
\tikzstyle directed=[postaction={decorate,decoration={markings,
    mark=at position .65 with {\arrow[arrowstyle]{stealth}}}}]
\tikzstyle reverse directed=[postaction={decorate,decoration={markings,
    mark=at position .65 with {\arrowreversed[arrowstyle]{stealth};}}}]

\def\vn{\vec{n}}
\def\J{\mathcal{J}}
\def\w{\omega} 
 
\def\t{\theta}
\def\e{\varepsilon} 
\def\U{\mathcal{U}}

\def\N{\mathbb{N}}
 
\def\SS1{\mathbb{S}^1}

\def\Z{\mathbb{Z}} 
\def\R{\mathbb{R}} 
\def\s{\mathfrak{s}}

\def\B{\mathcal{B}}

\def\G{\mathcal{G}}

\newtheorem{theorem}{Theorem}
\newtheorem{lemma}[theorem]{Lemma}
\newtheorem{proposition}[theorem]{Proposition}
\newtheorem{corollary}[theorem]{Corollary}

\DeclareMathSymbol{\varnothing}{\mathord}{AMSb}{"3F} 
\renewenvironment{proof}{\noindent {\bf Proof.}}{ \hfill\qed\\ }
\newenvironment{proofof}[1]{\noindent {\bf Proof of #1.}}{ \hfill\qed\\ }

\begin{document}

\title{Unique ergodicity for infinite area translation surfaces}  

\author {Alba M\'alaga Sabogal}
\address{Inria, CRI de Paris}
\email{alba.malaga@polytechnique.edu}

\def\curraddrname{{\itshape Address}}
\author{Serge Troubetzkoy}
\address{Aix Marseille Univ, CNRS, Centrale Marseille, I2M,  Marseille, France}

  \curraddr{I2M, Luminy\\ Case 907\\ F-13288 Marseille CEDEX 9\\ France}

 \email{serge.troubetzkoy@univ-amu.fr}
 
 \thanks{
 The project leading to this publication has received funding from Excellence Initiative of Aix-Marseille University - A*MIDEX and Excellence Laboratory Archimedes LabEx (ANR-11-LABX-0033), French "Investissements d'Avenir" programmes. Theorem \ref{thm:periodic} was inspired by a very nice lecture of Omri Sarig.}
\begin{abstract}   
We consider infinite staircase translation surfaces with varying step sizes.  For typical step sizes 
 we show that the translation flow is uniquely ergodic in almost every direction.
Our result also hold for typical 
configurations of the Ehrenfest wind-tree model endowed with the Hausdorff topology.
In contrast,  we show that the translation flow on a  periodic translation surface 
can not be uniquely ergodic in any direction.
\end{abstract} 
\maketitle

\section{Introduction}
One of the most fundamental results in the theory of (compact) translation surfaces is the theorem of Kerchoff, Masur and Smillie which states that the geodesic flow on any compact translation surface is uniquely ergodic in almost every direction \cite{KeMaSm}. The importance of this result is such that there are several articles explaining the proofs \cite{Ar,GoLa,Mo} as well as \cite{MaTa}).
Since the result holds for all translation surfaces, it holds for translation surfaces arising from billiards in 
rational polygons.

In the past decade there has been intensive study of translation surface and polygonal billiards with 
infinite area. In particular,  research has concentrated on trying  to understand if the ergodicity conclusion of the 
Kerchoff, Masur, Smillie theorem holds.  There are examples of infinite billiard tables/translation surfaces 
which are not ergodic in almost every direction with respect to the natural invariant area measure
 \cite{FrUl,FrHu}, and others which are ergodic in almost every direction with respect to the natural 
invariant area measure \cite{HoHuWe,HuWe,MSTr2,RaTr} (see also \cite{FrUl1} for some partial results in 
this direction);  certain of these examples are shown in Figure \ref{fig-1}.
There is a special class of (compact) translation surfaces called Veech surfaces, on a (compact) Veech 
surface each minimal direction is in fact uniquely ergodic \cite{Ve1}\cite{Ve2}; and thus one might think that 
they are
natural candidates for unique ergodicity in infinite area as well. This is not the case, it turns out that for the 
only known Veech example of an infinite area translation surface, the translation flow 
has many ergodic invariant Radon measures in almost every direction \cite{HoHuWe}. Hooper has classified
the invariant measures of the translation flow in many direction for certain other examples \cite{Ho}.

We begin by showing that the translation flow on a periodic translation surfaces can never be uniquely 
ergodic (Theorem \ref{thm:periodic}).
Our main result is that there are infinite area translation surfaces for which the area measure is the unique (up to scaling) invariant ergodic Radon measure for the translation flow in almost every direction.
 We present our results in two different settings,  first for infinite staircases (Theorem \ref{main-stair}), and 
 then  in the framework of the Ehrenfest wind-tree model (Theorem \ref{main}).  In both cases we show that 
 the generic surface satisfies this unique 
ergodicity  result.
 We remark that unique ergodicity has been observed for non-compact translation surfaces of finite area 
 by Hooper \cite{Ho} as well as by Rafi and Randecker \cite{RaRa}, however in this case we have 
 classical finite measure unique ergodicity.

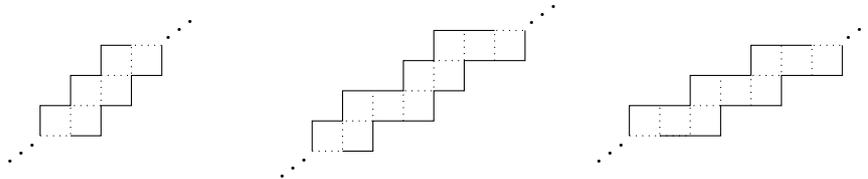
\begin{figure}[t]
\begin{minipage}[ht]{0.33\linewidth}
\centering
\begin{tikzpicture}[scale=0.4]
\draw[dotted] (0,0) -- (1,0) -- (1,1) -- (2,1) -- (2,2) -- (3,2) -- (3,3) -- (4,3);
\draw[] (1,0) -- (2,0) -- (2,1) -- (3,1) -- (3,2) -- (4,2) -- (4,3);
\draw[] (0,0) -- (0,1) -- (1,1) -- (1,2) -- (2,2) -- (2,3) -- (3,3);
\node at (-0.8,-0.3) {\fontsize{4}{6} \reflectbox {$\ddots$}};
 \node at (4.4,3.8) {\fontsize{4}{6} \reflectbox  {$\ddots$}};
\end{tikzpicture}
\end{minipage}\nolinebreak
\begin{minipage}[ht]{0.33\linewidth}
\centering
\begin{tikzpicture}[scale=0.4]
\draw[dotted] (0,0) -- (1,0) -- (1,1) -- (2,1) -- (2,2);
\draw[dotted] (3,1) -- (3,2) -- (4,2) -- (4,3) -- (5,3) -- (5,4);
\draw[dotted] (6,3) -- (6,4) -- (7,4);
\draw[] (1,0) -- (2,0) -- (2,1) -- (4,1) -- (4,2) -- (5,2) -- (5,3)-- (7,3) -- (7,4);
\draw[] (0,0) -- (0,1) -- (1,1) -- (1,2) -- (3,2) -- (3,3) -- (4,3) -- (4,4) -- (6,4);
\node at (-0.8,-0.3) {\fontsize{4}{6}  \reflectbox {$\ddots$}};
 \node at (7.4,4.8) {\fontsize{4}{6} \reflectbox  {$\ddots$}};
\end{tikzpicture}
\end{minipage}\nobreak
\begin{minipage}[ht]{0.33\linewidth}
\centering
\begin{tikzpicture}[scale=0.4]
\draw[dotted] (0,0) -- (1,0) -- (1,1) -- (2,1) -- (2,2);
\draw[dotted] (3,1) -- (3,2) -- (4,2) -- (4,3);
\draw[dotted] (-1,0) -- (0,0) -- (0,1);
\draw[dotted] (5,2) -- (5,3) -- (6,3);
\draw[] (0,0) -- (2,0) -- (2,1) -- (4,1) -- (4,2) -- (6,2) -- (6,3);
\draw[] (-1,0) -- (-1,1) -- (0,1) -- (1,1) -- (1,2) -- (3,2) -- (3,3) -- (5,3);
\node at (-1.8,-0.3) {\fontsize{4}{6} \reflectbox {$\ddots$}};
 \node at (6.4,3.8) {\fontsize{4}{6} \reflectbox  {$\ddots$}};
\end{tikzpicture}
\end{minipage}\nobreak
\caption{Periodic translation surfaces formed by identifying opposite sides. The first two are ergodic in almost every direction \cite{HoHuWe,RaTr}, while for the  third the  ergodic directions are of measure 0 \cite{FrUl}.}\label{fig-1}
\end{figure}

\section{Definitions and main results}
\subsection{Ergodic theory}
Let $T$ be a measurable map on a measurable space $(\Omega,\B)$, and suppose $\mu$ is a $\sigma$-finite measure on $(\Omega,\B)$ s.t. $\mu(\Omega) = \infty$.
We say that $\mu$ is \textit{invariant}, if $\mu(T^{-1}(E)) = \mu(E)$ for all $E \in \B$.
We say that $\mu$ is \textit{ergodic}, if for every set $E \in \B$ such that $T^{-1}(E) = E$ either 
$\mu(E) = 0$ or $\mu(\Omega \setminus E) = 0$.

Suppose $\Omega_0$ is a locally compact second countable metric space with
Borel $\sigma$-algebra $B_0$. Let $C_c(\Omega_0) := \{f : \Omega_0 \to \R: f$ continuous with compact support$\}$. 
A regular Borel measure $\mu$ on $\Omega_0$  is called a \textit{Radon measure}, if $\mu(C) < \infty$ for every compact set $C \subset \Omega_0$,
 
We will need to deal with Borel maps $T$ which are only defined on a subset $\Omega \subset \Omega_0$ with $\Omega \in \B_0$.
 
Let $\B := \{ E \cap \Omega : E \in \B_0\}$.  A measure $\mu$ on $(\Omega,\B)$  is
called \textit{locally finite}, if $\mu_0(E) := \mu(E \cap \Omega)$  is a Radon measure on $(\Omega_0,\B_0)$. 

We call the map $T$ \textit{uniquely ergodic} if
up to scaling,
$T$ admits a unique Radon  invariant measure.
  
\subsection{Periodic translation surfaces}
A translation surface $M$ is given by an at most countable index set $\Lambda$ and a 
disjoint collection of convex polygons 
$\{P_i \subset \R^2\}_{i \in \Lambda}$ with edges glued in pairs by translation.  The surface is 
compact if the set $\Lambda$ is finite.

Let $M$ be a compact translation surface;  $p: \tilde M  \to M $ is a  $\Z$-cover of $M$ of translation surfaces, i.e.,  $\tilde M$ a (noncompact) translation surface,
there is a finite set $P \subset M$ and a
covering map   $p: \tilde M \setminus p^{-1}(P) \to M \setminus P$  which is a triangulation in each
chart, and there is a translation automorphism $S : \tilde M \to \tilde M$ commuting with $p$, such that $M$
is isomorphic to $\tilde M / \langle S \rangle$. 

In a similar way,  $p: \tilde M  \to M $ is a  $\Z^2$-cover of $M$ if there is a finite set $P \subset M$ as above, and there are two translation automorphisms $S, S' : \tilde M \to \tilde M$, non aligned with each other and commuting with $p$, such that $M$
is isomorphic to $\tilde M / \langle S, S' \rangle$. 

We call a noncompact translation surface a \textit{periodic translation surface} if it is a $\Z^d$ cover of a compact translation surface ($d=1$
or $2$).

\begin{theorem}\label{thm:periodic}
For every direction $\theta$, the translation flow on a periodic translation surface in the direction $\theta$ is not uniquely ergodic.
\end{theorem}  

\subsection{Staircases}
A \textit{staircase translation surface} or simply a {staircase} is a translation surface obtained by gluing an enumerated, ordered, collection of same size rectangles, say 2 by 1, all of whose sides are parallel to the coordinate axes in the following way. We place a rectangle in the plane with the center of the rectangle at the origin, label it as the 0th rectangle.  The bottom of first rectangle will
intersect the top of the 0th rectangle with length of intersection $w_0 \in (0,1)$.  Choosing a sequence $w := \{w_i \in (0,1) : i \in \Z\}$, we continue this procedure inductively to produce a staircase like collection of rectangles. Then
we identify opposite sides of the staircase to form a translation surface  which we will call $S_{w}$ (see Figure \ref{fig0}). Note that corners of the rectangles give rise to conical points. Every such corner, considered as a point on the staircase surface, is a conical point with angle $6\pi$.

The set of all staircases is then coded by the set $(0,1)^\Z$.  We consider the product topology on this space, its closure is $[0,1]^\Z$, a Baire space.  In the proof we will also consider parameters $w \in [0,1]^\Z$. In the case $w_n = 0$ for some $n$, after removing the singular points  $S_w$ is not connected.  These non-connected staircases will play a very important role in the proof.

Fix a direction $\theta$ and consider the translation flow $\psi_t^{\theta}=\psi_t^{\w,\theta}$ on the surface $S_w$.  It will be convenient to use the  section of this flow defined by intersecting 
$S_{\w}$ with the collections of lines $y=n$ where $n \in \Z$, we identify this section with the set
$X := \Z \times [0,2)$. Note that in $S_{\w}$ the points $0$ and $2$ are identified, thus each  set $\{n\} \times [0,2)  \subset X$ is a circle.  After the identification, the set $X$ does not depend on the parameter $\w$,  but sometimes we need to emphasize the nature of this sets as phase spaces for dynamical systems,  we will then 
write $X^{\w}$.
Let $X^N = X^{\w,N} \subset X^\w$ denote the set $X^N := \{-N + 1,\dots, N\} \times [0,2)$. 
Likewise, the sets do not depend on the direction $\theta$ chosen,  when needed for clarity we sometimes 
write $X^{\theta}$ or $X^{\w,\theta}$.

Let $T^{\theta}=T^{\w,\theta}$ be the first return map of the flow $\psi_t^{\theta}$ to the section $X$. Note that $T^{\theta}$ preserves the length measure $\mu$, and that this measure is an infinite measure.

\begin{proposition}\label{prop:cons}
For any $\w$ such that $\lim_{i \to \pm\infty}w_i=0$, $(T^{\w,\t},X,\mu)$ is a conservative system for all $\theta$.
\end{proposition}
This proposition is closely related to results in \cite{Tr,MS,MSTr1,MSTr2}. We still provide a proof for completeness (\S \ref{sec:cons}).
 We remark that the set of $\w$ satisfying this condition is a $G_\delta$ set.

\begin{figure}[t]

\begin{tikzpicture}[scale=1]

 \node at (-0.4,-0.3) {\fontsize{4}{6} \reflectbox{$\ddots$}};
 \node at (2.5,1.5) {\tiny \textbullet};
 \draw[] (0.4,0) -- (2,0) -- (2,1)-- (3.5,1)--(3.5,2) --(4.5,2) -- (4.5,3) -- (6,3) -- (6,4);
  \draw[] (0,0) -- (0,1) -- (1.5,1)-- (1.5,2)-- (2.5,2) -- (2.5,3) -- (4,3) -- (4,4) -- (5.8,4);

\draw[densely dotted] (0,0) -- (0.4,0); 
\draw[densely dotted] (1.5,1) -- (2,1);
\draw[densely dotted] (2.5,2) -- (3.5,2);
\draw[densely dotted] (4,3) -- (4.5,3) ;
\draw[densely dotted] (5.8,4) -- (6,4) ;

 \node at (0.2,-0.2) {\fontsize{4}{6} $w_{^-2}$};
 \node at (1.7,0.8) {\fontsize{4}{6} $w_{^-1}$};
 \node at (2.85,1.8) {\fontsize{4}{6} $w_0$};
 \node at (4.25,2.8) {\fontsize{4}{6} $w_1$};
  \node at (5.8,3.8) {\fontsize{4}{6} $w_2$};
 \node at (6.3,4.5) {\fontsize{4}{6} \reflectbox{$\ddots$}};
 
 \draw[red] (0,0.5) -- (2,0.5);
\draw[red] (1.5,1.5) -- (3.5,1.5);
\draw[red] (2.5,2.5) -- (4.5,2.5);
\draw[red] (4,3.5) -- (6,3.5);
 
  \node at (2.8,0.5) {\fontsize{4}{6}   $\{-1\} \times [0,2)$};
 \node at (4.3,1.5) {\fontsize{4}{6} $\{0\} \times [0,2)$};
  \node at (5.3,2.5) {\fontsize{4}{6} $\{1\} \times [0,2)$};
   \node at (6.8,3.5) {\fontsize{4}{6} $\{2\} \times [0,2)$};
   
\end{tikzpicture}
\caption{The staircase, opposite sides are identified. The section $X$ is marked in red.}\label{fig0}

\end{figure}
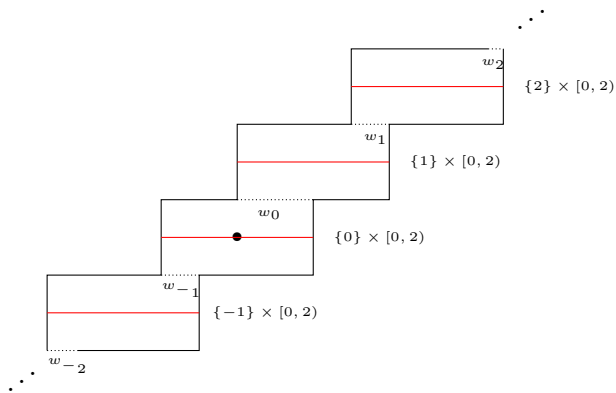

\begin{theorem}\label{main-stair} There is a dense $G_\delta$ subset $G$ of $[0,1]^{\Z}$ and a dense $G_\delta$ set of full measure of directions $\mathcal H$,  such that for each $w \in G$, we have $w \in (0,1)^{\Z}$ and, up to scaling,  $\mu$ is the unique $T^{\theta}$-invariant ergodic Radon measure, for every  $\theta \in \mathcal H$.
\end{theorem}
It follows that, up to scaling, the area measure is the unique $\psi_t^{\theta}$-invariant ergodic Radon measure, for every  $\theta \in \mathcal H$.

\subsection{Wind-trees}
In the wind-tree model a point particle moves without friction on the plane with at most infinitely many rhombi removed, and collides elastically with the rhombi.  This model was introduced by Paul and Tatiana Ehrenfest in 1912 in their
 a seminal article on the foundations of Statistical Mechanics in order to interpret the work of Boltzmann and Maxwell on gas dynamics \cite{EhEh}.

A square whose sides are parallel to lines $y=\pm x$ will be referred to as \textit{rhombus} or \textit{tree}. The $\mathcal L_1$ distance in $\mathbb R^2$ will be denoted by $d$. Note that balls with respect to this distance are rhombi.

Fix $s > 0$. A \textit{configuration} is an at most countable collection of rhombi with diameter $s$, whose interiors are pairwise disjoint. 
Since $s$ is fixed it is enough to note the centers of the rhombi, thus 
a configuration $g$ is  an at most countable subset of $\mathbb R^2$ such that if $z_1,z_2 \in g$ then $d(z_1,z_2) \ge s$.

To define a topology on the set of configurations consider polar coordinates $(r,\theta)$ on the plane. Each point $(r,\theta)$ in the plane is the stereographic projection of a point in the sphere with spherical coordinates $(2\arctan(1/r),\theta)$.   
Apply  the inverse of the stereographic projection to a configuration $g$ to obtain a subset of the sphere.  
Let $\hat{g}$ denote the union of this  set with the north pole of the sphere denoted by  $\{\infty\}$, it is a closed subset of the sphere.
The topology we define on the set of configurations is then induced by the Hausdorff distance $d_H$ on the sphere
given by  
$$d_H(g_1,g_2) = \max(\sup\limits_{z_1\in \hat{g}_1} \inf\limits_{z_2\in \hat{g}_2} \rho(z_1,z_2) , \sup\limits_{z_2\in \hat{g}_2} \inf\limits_{z_1\in \hat{g}_1} \rho(z_1,z_2)).$$
Here $\rho$ denotes the geodesic distance on the sphere, i.e., the length of the shortest path from one point to another along the great circle passing through them.  

Let $\textit{Conf}$ be the set of all configurations.
In  \cite{MSTr4}, we show that $(\textit{Conf},d_H)$ is a compact metric space, thus a Baire space.

Fix $g \in \textit{Conf}$.  The \textit{wind-tree table} $\B^g$ is the plane $\mathbb{R}^2$ with the interiors of the union of the trees
removed.
Fix  $\theta \in \SS1$.
The \textit{billiard flow $\phi_t^{g,\theta}$ in the direction $\theta$}   is  the free motion in the interior of $\B^g$ with elastic collision from the boundary of $\B^g$ (the boundary of the union of the trees).
If a billiard orbit hits a corner of a tree, the outcome of the collision is not defined, and we agree that the billiard orbit stops there, its future is not defined anymore.
Let $T^{g,\theta}$ be the first return map of the flow  $\phi_t^{g,\theta}$ to the boundary. 
We will sometime suppress the
$g$ dependance of certain notations if it can be deduced from the formulae, for example we might write $T^{\theta}$ for $T^{g,\theta}$.

Throughout the wind-tree argument we will always suppose that $\theta$ is not parallel to the rhombi sides. Once launched in the direction $\theta$, the billiard direction can achieve four directions $[\theta] := \{\theta, \frac32\pi - \theta, \pi + \theta, \frac12\pi-\theta \}$;
thus the \textit{phase space} $X^{g,\theta}$  
of the billiard map in the direction $\theta$ is the set of pairs $(\s,\phi)$ such that at $\s \in \partial \B^g$ the
direction $\phi$ points to the interior of the table, i.e. away from the trees.
Note that in this notation  $T^\theta$,  $T^{\frac32\pi - \theta}$, $T^{\pi + \theta}$ and $T^{\frac12\pi-\theta}$
are all the same.

For each direction $\theta$, the billiard map  $T^{\theta}$ preserves 
 the length measure $\mu$ on $\partial \B^g$ times the discrete measure on $[\theta]$, we
will also call this measure $\mu$. 
Note that $\mu$ is an infinite measure.

\begin{theorem}\label{main} For any $s>0$ there is a dense $G_\delta$ subset $G$ of $\textit{Conf}$ and a dense $G_\delta$ set of full measure of directions $\mathcal H$,  such that for each $g \in G$, up to scaling, the length measure $\mu$ is the unique $T^{\theta}$-invariant ergodic measure, for every  $\theta \in \mathcal H$.
\end{theorem}

Note that for each $g \in G$, up to scaling, the area measure is thus the unique $\varphi_t^{g,\theta}$-invariant ergodic measure, for every  $\theta \in \mathcal H$.

\subsection{Generalizations}	
Our results hold in a much more general setting.  In particular they hold in any Baire space $X$ of infinite area translation surfaces such
that
\begin{enumerate}
\item  the set  \{$S \in X: $  there is a compact
translation surface $S^{comp} \subset S\}$ contains a dense countable set $\mathcal{S}$,
\item suppose $(S_i) \subset \mathcal{S}$ converges sufficiently quickly to a surface $S \not \in \mathcal{S}$, then $S^{comp}_i$ converges to $S$,
and $S$ is connected.
\end{enumerate}
The proof of our results in this more general setting are essentially identical to the proofs of Theorems \ref{main-stair} and \ref{main},
we choose not to present an abstract proof for readability issues.

We list some examples of other situations where our theorems hold.

1) Billiards.
Theorem \ref{main-stair} also holds if we change the staircase dynamics to billiard dynamics, i.e., we do not identify the opposite sides of the staircase,  and the flow 
collides elastically when reaching the boundary.  

More generally fix an arbitrary rational polygon $P$.  We can construct a staircase like billiard table from $P$, which we call an
{\it infinite polygon}, by choosing
two sides of $P$ and a sequence of lengths $\w_n \in [0,a]$ where $a > 0$ is the shorter of the two lengths of the chosen sides. 
We then construct the infinite polygon in an analogous way to the construction of staircase.
We can also start with a countable collection of rational polygons $P_n$, each with a pair of marked sides $(s_1^n,s_2^n)$. 
Let $a_n$ be the minimum of the lengths of of $s_2^n,s_1^{n+1}$.  Then glue them as above with intersection $\w_n \in [0,a_i]$. 

2) Generalized staircases (see Figure \ref{fig:gs}).
Take a countable enumerated collection of compact translation surfaces $S_n$.  Consider a sequence $\w_n \in [0,1]$. On $S_n$ we draw two parallel, 
non-intersecting segments of length $\w_n$ and $\w_{n+1}$.  Then we
construct an infinite generalized staircase by gluing the slits of length $\w_n$ in the following sense, i.e., when the flow arrives at a slit
in $S_n$, then it jumps to the glued slit in $S_{n+1}$ and continues it the same direction. To make the notion of same direction precise 
we consider a polygonal representation of  each $S_n \subset \R^2$ with the slits horizontal, then if we hit a slit from the bottom on $S_n$
we exit from the top of the slit on $S_{n+1}$.

\begin{figure}[h]

\begin{tikzpicture}[scale=1]

 \node at (-0.4,-0.3) {\fontsize{4}{6} \reflectbox{$\ddots$}};
 \draw[] (0.4,0) -- (2,0) -- (2,1)-- (3.5,1)--(3.5,2) --(4.5,2) -- (4.5,3) -- (6,3) -- (6,4);
  \draw[] (0,0) -- (0,1) -- (1.1,1) -- (1.1,1.2) --(1.2,1.2) -- (1.2,1) -- (1.5,1)-- (1.5,2) --(1.6,2) --(1.6,2.3) -- (1.7,2.3)  -- (1.7,2.4) -- (1.9,2.4) -- (1.9,2) -- (2.5,2) -- (2.5,3) --
(2.65,3) -- (2.65,3.3) -- (3.7,3.3) -- (3.7,3)-- (4,3) -- (4,4) -- (4.1,4) -- (4.1,4.2) -- (4.3,4.2) -- (4.3,4.15)  -- (4.4,4.15) -- (4.4,4)-- (5.8,4);

\draw[densely dotted] (0,0) -- (0.4,0); 
\draw[densely dotted] (1.5,1) -- (2,1);
\draw[densely dotted] (2.5,2) -- (3.5,2);
\draw[densely dotted] (4,3) -- (4.5,3) ;
\draw[densely dotted] (5.8,4) -- (6,4) ;

 \node at (0.2,-0.2) {\fontsize{4}{6} $w_{^-2}$};
 \node at (1.7,0.8) {\fontsize{4}{6} $w_{^-1}$};
 \node at (2.85,1.8) {\fontsize{4}{6} $w_0$};
 \node at (4.25,2.8) {\fontsize{4}{6} $w_1$};
  \node at (5.8,3.8) {\fontsize{4}{6} $w_2$};
 \node at (6.3,4.5) {\fontsize{4}{6} \reflectbox{$\ddots$}};
   
\end{tikzpicture}
\caption{A generalized staircase (obvious identifications hold)}\label{fig:gs}

\end{figure}
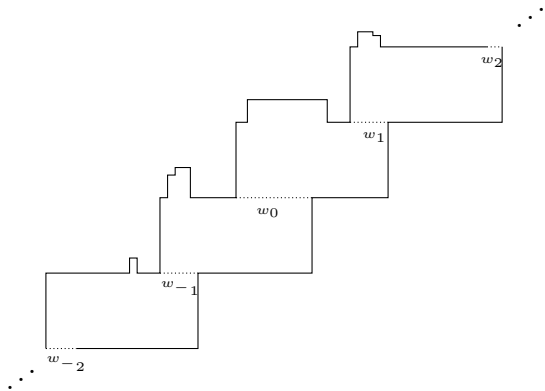

3. Wind-trees. In two previous articles we have considered a subset of configurations which are small perturbations of lattice configurations, and we showed  that the generic wind-tree is minimal and ergodic in  almost every direction \cite{MSTr1,MSTr2}.  The proofs of these two results
hold mutatis  mutandis  in the more general settings (staircases and wind-trees) which we consider here.  The results from the present paper hold in this setting as well.

\subsection{Comparison to other definitions of unique ergodicity}
In his survey article \cite{Sa} Sarig defines two notions of unique ergodicity in the infinite measure setting which are stronger than the one we prove.

A point $x \in \Omega$ is called {\it generic} for $\mu$  if
 for all $f,g \in C_c(\Omega)$ such that $g \ge 0$ and $g \ge 0$ and $\int g \, d\mu > 0$ we have
  $$ \frac{ \sum_{k=0}^{\ell} f \big (T^k(x)  \big)}{\sum_{k=0}^{\ell} g \big (T^k( x)  \big)} \to  \frac{\int_{\Omega} f(y) \, d\mu(y)}{\int_{\Omega} g(y) \, d\mu(y)}.\\$$
  
Sarig calls a map $T$ \textit{uniquely ergodic} if
 (1) up to scaling,
$T$ admits a unique Radon  invariant measure not supported on a
single orbit; and (2) every point is generic for
this measure.

This strong version of unique ergodicity does not hold in either of our settings. In fact, in \cite{MSTr1} we showed that
 the existence of non-recurrent points, this result holds in all the settings we consider here, thus (2) can not hold.

Sarig calls a map $T$ \textit{uniquely ergodic in the broad sense } if (1) up to scaling,
$T$ admits a unique Radon ergodic invariant measure not supported on a
single orbit; and (2) every non-exceptional non-periodic point is generic for
this measure (see \cite{Sa}).

Here  a  point
$x$ is called (forward) \textit{exceptional} for a map $T$  if the measure $\sum_{n > 0} \delta_{T^n(x)}$
is locally finite, where $\delta_y$ denote the point mass at $y$.

We do not know if  the proofs of Theorems \ref{main-stair} and \ref{main} can be strengthened to show unique ergodicity in the broad sense.

\section{Non-unique ergodicity}
Let $\tau : I \to I$ be a uniquely ergodic (finite) interval exchange transformation. Let $d=1$ or $2$ and
let $\psi : I \to \Z^d$ be a function which is locally constant, with a finite set of discontinuities.  Let $X = I \times \Z^d$ and consider the
skew product $T: X \to X$ defined by $T(x,\vn) = (\tau(x), \vn + \psi(x))$.

The map $T$ is an infinite interval exchange transformation.  
Note that for any $\Z$ or $\Z^2$ periodic translation surface, the first
return map will be of this form for almost every direction.  This is in particular true for 
the periodic staircases and periodic wind tree models.

Let $h : I \to \R_+$ be a Borel measurable map.  A Borel probability measure $\mu$ on $I$  is called an 
\textit{$(h,\tau)$-conformal measure} if $\mu \circ \tau$ is absolutely continuous with respect to $\mu$ 
and the Radon-Nikodym derivative satisfies $\frac{d\mu \circ \tau}{d\mu}(x) = h(x)$ for $\mu$-a.e.\ $x \in I$.

Let $\chi : \Z^d \to \R$ 
be a group homomorphism, and $\mu_{\chi}$ a $(e^{\chi \circ \psi},\tau)$-conformal
measure.  For each $\vn \in \Z^d$, consider the map $\pi_{\vn} : I \to I \times \Z^d$ given by
$\pi_{\vn}(x) = (x,\vn)$. A measure $m_\chi$ on $I \times \Z^d$ is called the
$\chi$\textit{-Maharam measure} associated to $\tau$, $\mu_{\chi}$ and $\chi$ if  it satisfies the equation
 $m_{\chi} \circ \pi_{\vn} = \frac{1}{e^{\chi(\vn)}}
\mu_{\chi}.$

\begin{lemma}
In this context Maharam measures always exist, are locally finite, and are invariant under the skew product transformation $T$.
\end{lemma}

\begin{proof}
Passing to the Cantor representation (refined so that $\psi$ is continuous), and applying
Theorem 4.1 in \cite{Sc} yields that $(e^{\chi\circ \psi},\tau)$-conformal measure exists in the Cantor 
representation.  The push down $\mu_{\chi}$ of this measure to $X$  is also a conformal measure.
The Maharam measure $m_\chi$  associated to $\tau$, $\mu_{\chi}$ and $\chi$ is given by the following,
$$dm_{\chi}(x,\vn)  := e^{-\chi(\vn)} d\mu_{\chi}(x). $$
Maharam measures are locally finite since 
$$m_\chi(I \times \{\vn\}) =e^{-\chi{(\vn)}}\times \mu_{\chi}(I)  = e^{-\chi{(\vn)}}< \infty$$ for any $\vn \in Z^d$.

The measure is $T$-invariant since the scaling factor of $\chi$ from $(e^{\chi \circ \psi},\tau)$-conformality
 cancels with the $\frac{1}{e^{\chi}}$ factor arising from the change in the $\Z^d$-coordinate.
\end{proof}

\begin{proofof}{Theorem  \ref{thm:periodic}}
Suppose that the direction is such that $\tau$ is not uniquely ergodic.
The product of each $\tau$-ergodic measure with the counting measure on $\Z^d$
is a $T$-invariant locally finite measure, these measures  are not rescalings of one another.

Suppose now that  $\tau$ is uniquely ergodic. For each $a  \in \R$ consider the group homomorphism  $\chi_a(n) = na$ in the case $d=1$ and
$\chi_a(\vn) = a(n_1+ n_2)$ in the case $d=2$ with $\vn= (n_1,n_2)$.
Consider the associated
$(e^{\chi_a \circ \psi},\tau)$-conformal measure ${\mu_a}$, it satisfies $\frac{d\mu_a \circ \tau}{d\mu_a}(x) = e^{\chi_a \circ \psi (x)}$ 
for $\mu_a$-a.e.\ $x \in I$
and
the associated Maharam measure $m_a$ satisfies
$$dm_{a}(x,\vn)  := e^{-\chi_a(\vn)} d\mu_a(x). $$
The measure $m_a$ is a locally finite measure.
Suppose that $m_a$ is ergodic for at least two values of $a$, 
then it suffices to notice that if $a \ne b$ then $\chi_a$ and $\chi_b$ are  not rescalings of one another. 
On the other hand if $m_{a}$ is not ergodic for some (all but one)  $a \in \R$ then, since the ergodic decomposition of a locally finite measure is into locally finite measures we are also done.
\end{proofof}

The general program to study Maharam measures  was proposed in  a much more general 
setting in \cite{AaNaSaSo}.  The aim of this program  is to classify all the ergodic measures of the 
associated skew products, this is in general difficult and has been achieved only for special base maps.
In particular, in the frame work of translation surfaces such a classification has been achieved 
for certain infinite translation surfaces in \cite{HuWe} and \cite{Ho}.

 \section{Proof of the staircase result}\label{secStair}
 
\begin{proofof}{Theorem \ref{main-stair}} 
 Throughout the proof we will occasionally confound the parameter $\w$ with the staircase $S_\w$.

The strategy of the proof is as follows: we choose a dense set $\{S_{\w^i}\}$ of staircases which satisfy the goal dynamical property of unique ergodicity in almost every direction on  compact sets which exhaust the staircase.  The parameter of these  staircases will satisfy $\w \in [0,1)^\Z \setminus (0,1)^\Z$.
Then we will show that staircases which are sufficiently well approximated by this dense set will satisfy 
 the  goal dynamical property on the whole phase space.  

A staircase $S_\w$ is \textit{$N$-ringed}  if $w_N = w_{-N} = 0$ and  $w_j \not \in \{0,1\}$ for all $|j| < N$ (the name ringed comes from the corresponding topologically more complicated definition for wind-tree
configurations). We will consider the dynamics on the set $X^N$, i.e., inside the ring, the set $X \setminus X^N$ 
plays no role in our proof.

 Let $\{\w^i\}$ be a dense set of parameters such that each $S_{\w^i}$ is an $N_i$-ringed configuration, with $N_i$ increasing with $i$.
 
 By \cite{KeMaSm} the translation flow is uniquely ergodic (in the classical finite measure sense) in almost every direction inside the ring, and thus  the  map $T^{\w^i,\theta}|_{X^N}$
is also uniquely ergodic in almost every direction. To transfer this finite area unique ergodicity to sufficiently well approximable generic staircases, we suppose that $\delta_i$ are strictly positive numbers and consider the dense $G_\delta$ set
$$ 
\G := \bigcap_{m = 1}^{\infty}   \bigcup_{i=m}^{\infty} \U_{\delta_i}(\w^i).
$$
We will show that the $\delta_i$ can be chosen in such a way
that $\G \subset (0,1)^\Z$ (i.e., the staircases in $\G$ are connected), and all the staircases in $\G$ have a unique (up to scaling) Radon $T^{\theta}$-invariant ergodic measure for all $\theta\in\mathcal H$, where $\mathcal H$ is a $G_\delta$ set of full measure that will be constructed in the proof. 

For each $N \in \N$ let  
$$\{h^{N}_j: X^N \to \R\}_{j \ge 1} $$ 
be a countable collection of continuous, strictly positive functions  which are dense
with respect to the sup norm in the set of all continuous non-negative functions.
We can think of $h^N_j$ as a function with compact support defined on  $X$.
We  choose an enumeration  $\{h_j\}_{j \ge 1}$  of the set $\cup_{j,n} \{h^N_j\}$.  By construction, this
collection is dense with respect to the sup norm in the set of all continuous non-negative functions on $X$.

Fix a surface $S_{\w}$, a direction $\theta \in \SS1$,
a point $z \in X$, and $(j,n) \in \N^2$. Consider the Hopf average
$$H^{\w,\theta}_{j,n,\ell}( z)  := \frac{ \sum_{k=0}^{\ell} h_j \big ((T^{\w,\theta})^k( z)  \big)}{\sum_{k=0}^{\ell} h_n \big ((T^{\w,\theta})^k( z)  \big)}.$$

Our strategy is to study unique ergodicity via Hopf averages.  We start with the ringed configuration $\w^i$. We consider the Hopf averages of a finite subcollection of the $\{h_j\}$ and the times $\ell_i$ where these averages almost converge. We will show that for a small perturbation $S_\w$ of $S_{\w^i}$, for many directions, for each point $z$ either the forward or the backward
orbit segment of length $\ell_i$ stays close enough to a good orbit on $S_{\w^i}$ to control the Hopf averages on $S_\w$.  Since for a dense $G_\delta$ this behavior happens on infinitely many scales, we can  conclude that the forward or the backward Hopf average converges to 
the ratio of the integrals with respect to Lebesgue measure.   Then we apply  
the following  criterion for unique ergodicity.

\begin{lemma}\label{lemma6}
Fix $\w \in (0,1)^{\Z},\theta \in \mathbb{S}^1$. Suppose that there exists an increasing sequence $\ell_k$ tending to infinity such that   for each $z \in X$ either
\begin{eqnarray*}
\lim_{k \to \infty} H^{\w,\theta}_{j,n,\ell_k}(z) & =  & \frac{\int_{X} h_j(y) \, d\mu(y)}{\int_{X} h_n(y) \, d\mu(y)} \  \forall  
 (j,n) \in \N^2  \hbox{ or }\\
\lim_{k \to \infty} H^{\w,\theta}_{j,n,-\ell_k}(z) & = &  \frac{\int_{X} h_j(y) \, d\mu(y)}{\int_{X} h_n(y) \, d\mu(y)} \  \forall  (j,n) \in \N^2,
\end{eqnarray*}
then, 
up to scaling, there exists at most one conservative, $T^{\w,\theta}$-invariant, ergodic, Radon measure; if it exists it is the measure $\mu$.
\end{lemma}

\begin{proof}
Suppose that $m$ is a conservative, $T^{\w,\theta}$-invariant, ergodic Radon measure.  Fix $n$, by assumption we have $\int_{X} h_n(y) \, d\mu(y) > 0$ and $\int_{X} h_n(y) \, dm(y) > 0$.
By the Hopf ergodic theorem, for every $j \in \N$ and  $m$-a.e.\ $z \in X$ we have
$$\lim_{\ell \to \infty} H^{\w,\theta}_{j,n,\ell}(z)  =  \lim_{\ell \to \infty} H^{\w^i,\theta}_{j,n,-\ell}(z) = 
\frac{\int_{X} h_j(y) \, dm(y)}{\int_{X} h_n(y) \, dm(y)} .$$ 
But
by the assumptions of this lemma, this number  must coincide with $\frac{\int_{X} h_j(y) \, d\mu(y)}{\int_{X} h_n(y) \, d\mu(y)} $. 
Fix some $n \in \N$, then for each $j \in \N$ we have
\begin{equation}\label{proporsc}
\int_{X} h_j(y) \, dm(y)
=
C \int_{X} h_j(y) \, d\mu(y)
\end{equation}
with 
$$C = \frac{\int_{X} h_n(y) \, dm(y)}{\int_{X} h_n(y) \, d\mu(y)}.$$

Consider any  interval $Y := \{k\} \times \hat Y  \subset X$.  
We apply the triangle inequality.
\begin{align}\label{trisc}
& |m(Y) - C \mu(Y) | < \\ \nonumber
& \quad \quad \quad \quad
 \big | m(Y) - \int_{X} h_{j}(y) \, dm(y)  \big |  +\\ \nonumber
&\quad \quad \quad \quad \quad \quad
 \big |   \int_{X} h_{j}(y) \, dm(y)  -  C \int_{X} h_{j}(y) \, d\mu(y)\big |  + \\ \nonumber
&  \quad \quad \quad  \quad \quad \quad \quad \quad
 \big |  C \int_{X} h_{j}(y) \, d\mu(y) - C \mu(Y) \big | . \nonumber
 \end{align}

By the assumption of the density of the functions $\{h_j\}$ we can find a subsequence $\{h_{j_k}\}$ so that $\lim_{k \to \infty} h_{j_k} (x) = 1$ if $x \in int(Y)$ and $0$ if $x$ is not in the closure
of $Y$.  Thus since $m,\mu$ are Radon measures, for each $\e > 0$ we can choose $j_0$ such that

$$\max \Big (\big |m(Y) - \int_{X} h_{j_0}(y) \, dm(y) \big  |,  C \big |\mu(Y) - \int_{X} h_{j_0}(y) \, d\mu(y)\big  | \Big )< \e.$$
Combining this with Equations \eqref{proporsc} and \eqref{trisc} yields
$ |m(Y) - C \mu(Y) | < 2\e$.  
Since $\e > 0$ is arbitrary we conclude that $m = C \mu$.
\end{proof}

Notice that in the above lemma, the only rigidity we proved is with respect to conservative ergodic measures. But, the only $T^{\w,\theta}$-invariant ergodic measure on $X$  which are not conservative, are measures supported on a single, bi-infinite non-periodic orbit of $T^{\w,\t}$. Such a measure is not a Radon measure if the bi-infinite orbit is dense because every neighborhood is visited an infinite number of times. Thus combining Lemma \ref{lemma6} and Proposition \ref{prop:cons} yields

\begin{corollary}
If the assumptions of Lemma \ref{lemma6} are verified, then up to scaling, the length measure $\mu$ is the unique 
$T^{\w,\theta}$-invariant, ergodic Radon measure.
\end{corollary}

Thus for our proof we need to construct the set $\mathcal{H}$ and show that
 for every $\w \in \G$ the convergence suppositions of the Lemma hold for $\theta \in \mathcal{H}$, and  all $(j,n) \in \N^2$.

Let $J_i \in \N$ be the maximal $J \in \N$ such that for all $1 \le j \le J$ the support of $h_j$ is contained in $X^{N_i}$, i.e., inside the 
 the ring of $S_{\w^i}$. Let $\J_i := \{1,\dots,J_i\}^2$. Notice that  $J_{i+1} \ge J_i$ and that $\lim_{i \to \infty} J_i = \infty$.

For each $i$ let $A_i$ be the set of directions for which the map $T^{\w^i,\theta}$ is  uniquely ergodic when 
restricted to $X^{N_i}$. 
By \cite{KeMaSm} the set $A_i$ is of full measure. We additionally assume that  $A_i$ does not contain a saddle connection direction, this
removes an at most countable set from the set of uniquely ergodic directions.
Let $\gamma_i > 0$ be a sequence tending to $0$ and fix $i$.  We apply the Corollary \ref{corcor} of  the appendix on the uniform convergence of
Hopf averages  to any  $\w^i$, $\theta \in A_i$,  $(j,n) \in \J_i$
 to conclude that there is an  $\hat \ell$ which
depends on $i,\theta,j,n$ such that 
\begin{equation}
\Big |H^{\w^i,\theta}_{j,n,\ell}(z) -\frac{\int_{X} h_j(y) \, d\mu(y)}{\int_{X} h_n(y) \, d\mu(y)} \Big | < \gamma_i\label{e}
\end{equation}
for all $\ell \ge \hat \ell$ and all  $z \in  X^{N_i}$ (except those $z$ whose orbit hits a singular point before time $\ell$).
Choose $\hat \ell$ sufficiently large so that \eqref{e}
holds for all $(j,n) \in  \J_i$; note that $\hat \ell$ depends only on $\theta \in A_i$ and $i$.

Next we uniformize this estimate to a large set of directions. 
 We choose  ${\ell_i} \ge N_i$ and sets $B_i \subset A_i \subset \SS1$  so that $\lambda(B_i^c) < \gamma_i$  and 
 $\hat \ell(i,\theta) \le \ell_i$ for all $\theta \in B_i$. 
 Consider the finite set of saddle connections of length at most
 $2 \ell_i$, and a small open neighborhood of this set. By choosing the open set very small and 
 modifying $B_i$ we can additionally assume that 
it does not intersect this neighborhood
 
Now we would like to extend these estimates to the neighborhood $\U_{\delta_i}(\w^i)$
for a sufficiently small strictly positive $\delta_i$. We will impose various conditions on $\delta_i$ in an incremental way.
The first requirement is
\begin{itemize}
\item[\bf A0] $\delta_i< \min \{\w^i_j: -n_i < j < n_i\}$.
\end{itemize}
This requirement ensures that for all $\w \in \U_{\delta_i}(\w^i)$ we have $\w_j > 0$ for all $-n_i < j < n_i$ which will
ensure that the tables in the dense $G_\delta$ set $G$ are connected.

We will explicitly state the next three requirements, {\bf A1 - A3} for forward orbits, the analogous requirements are also assumed  for
backward orbits, but will not be stated explicitly.
Our goal is now to define a small neighborhood $C_i$ of $B_i$, and then
 a large open subset $D_i$ of $C_i$, and for each $\theta \in D_i$ to define a map
$\zeta^+ = \zeta^{+}(\w^i,\w,\theta)$ defined on a large subset of $X^{N_i}$ (for $\w$) onto $X^{N_i}$ (for $\w^i$) which is close to the identity, and a map $\zeta^-$ satisfying
similar conditions. The $+$ maps will control the future behavior of orbits, the $-$ map will control the past.  In one of our first articles on the Ehrenfest model where we studied  ergodicity \cite{MSTr2}  it was sufficient to use the identity map for the maps $\zeta^\pm$. In our article on infinite ergodic index \cite{MSTr4} we 
 we have defined a related map, but the form of the map considered in that article
is not adapted to the study of unique ergodicity.

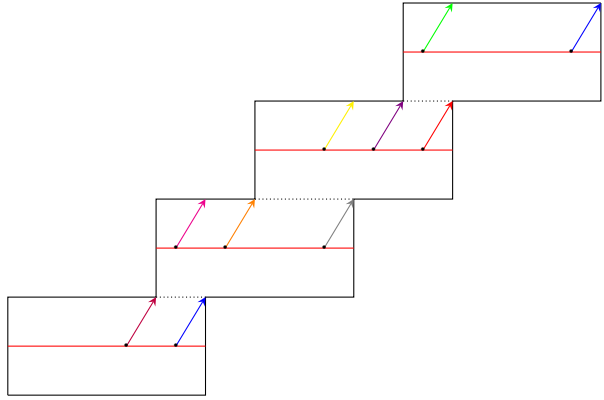
\begin{figure}[h]
\begin{tikzpicture}[scale=1.3]

 \draw[] (0,0) -- (2,0) -- (2,1)-- (3.5,1)--(3.5,2) --(4.5,2) -- (4.5,3) -- (6,3) -- (6,4);
  \draw[] (0,0) -- (0,1) -- (1.5,1)-- (1.5,2)-- (2.5,2) -- (2.5,3) -- (4,3) -- (4,4) -- (6,4); 
\draw[densely dotted] (1.5,1) -- (2,1);
\draw[densely dotted] (2.5,2) -- (3.5,2);
\draw[densely dotted] (4,3) -- (4.5,3) ;
\draw[red] (0,0.5) -- (2,0.5);
\draw[red] (1.5,1.5) -- (3.5,1.5);
\draw[red] (2.5,2.5) -- (4.5,2.5);
\draw[red] (4,3.5) -- (6,3.5);

\draw[blue, ->, >=stealth] (1.7,0.5) -- (2,1);
\draw[purple, ->, >=stealth] (1.2,0.5) -- (1.5,1);
\draw[gray, ->, >=stealth] (3.2,1.5) -- (3.5,2);
\draw[orange, ->, >=stealth] (2.2,1.5) -- (2.5,2);
\draw[magenta, ->, >=stealth] (1.7,1.5) -- (2,2);
\draw[red, ->, >=stealth] (4.2,2.5) -- (4.5,3);
\draw[violet, ->, >=stealth] (3.7,2.5) -- (4,3);
\draw[yellow, ->, >=stealth] (3.2,2.5) -- (3.5,3);
\draw[green, ->, >=stealth] (4.2,3.5) -- (4.5,4);
\draw[blue, ->, >=stealth] (5.7,3.5) -- (6,4);

\node at (1.7,0.5) {$\cdot$};
\node at (1.2,0.5) { $\cdot$};
\node at (3.2,1.5) { $\cdot$};
\node at (2.2,1.5) { $\cdot$};
\node at (1.7,1.5) { $\cdot$};
\node at (4.2,2.5) { $\cdot$};
\node at (3.7,2.5) { $\cdot$};
\node at (3.2,2.5) { $\cdot$};
\node at (5.7,3.5) { $\cdot$};
\node at (4.2,3.5) { $\cdot$};

\end{tikzpicture}
\caption{The set $\Sigma^{\w^i,\theta,N_i}$.}\label{figa}
\end{figure}

\begin{figure}[h]
\begin{tikzpicture}[scale=1.3]

 \draw[] (0,0) -- (2,0) -- (2,1)-- (3.5,1)--(3.5,2) --(4.5,2) -- (4.5,3) -- (6,3) -- (6,4);
  \draw[] (0,0) -- (0,1) -- (1.5,1)-- (1.5,2)-- (2.5,2) -- (2.5,3) -- (4,3) -- (4,4) -- (6,4); 
\draw[densely dotted] (1.5,1) -- (2,1);
\draw[densely dotted] (2.5,2) -- (3.5,2);
\draw[densely dotted] (4,3) -- (4.5,3) ;
\draw[red] (0,0.5) -- (2,0.5);
\draw[red] (1.5,1.5) -- (3.5,1.5);
\draw[red] (2.5,2.5) -- (4.5,2.5);
\draw[red] (4,3.5) -- (6,3.5);

\draw[blue, ->, >=stealth] (1.7,0.5) -- (2,1);
\draw[purple, ->, >=stealth] (1.2,0.5) -- (1.5,1);
\draw[gray, ->, >=stealth] (3.2,1.5) -- (3.5,2);
\draw[orange, ->, >=stealth] (2.2,1.5) -- (2.5,2);
\draw[magenta, ->, >=stealth] (1.7,1.5) -- (2,2);
\draw[red, ->, >=stealth] (4.2,2.5) -- (4.5,3);
\draw[violet, ->, >=stealth] (3.7,2.5) -- (4,3);
\draw[yellow, ->, >=stealth] (3.2,2.5) -- (3.5,3);
\draw[green, ->, >=stealth] (4.2,3.5) -- (4.5,4);
\draw[blue, ->, >=stealth] (5.7,3.5) -- (6,4);

\node at (1.7,0.5) {$\cdot$};
\node at (1.2,0.5) { $\cdot$};
\node at (3.2,1.5) { $\cdot$};
\node at (2.2,1.5) { $\cdot$};
\node at (1.7,1.5) { $\cdot$};
\node at (4.2,2.5) { $\cdot$};
\node at (3.7,2.5) { $\cdot$};
\node at (3.2,2.5) { $\cdot$};
\node at (5.7,3.5) { $\cdot$};
\node at (4.2,3.5) { $\cdot$};

\draw[blue, dashed] (5.4,3) -- (5.7, 3.5);
\draw[blue, dashed] (5.4,4) -- (5.1, 3.5);
\node at (5.1,3.5) {$\cdot$};
\draw[blue, dashed] (5.1,3.5) -- (4.8,3);
\draw[blue, dashed] (4.8,4) -- (4.5,3.5);
\node at (4.5,3.5) {$\cdot$};

\draw[green, dashed] (4.2,3.5) -- (4,3.15);
\draw[green, dashed] (6,3.15) -- (5.9,3);
\draw[green, dashed] (5.9,4) -- (5.3,3);
\draw[green, dashed] (5.3,4) -- (5,3.5);
\node at (5,3.5) {$\cdot$};
\node at (5.6,3.5) {$\cdot$};

\draw[red, dashed] (4.2,2.5) -- (3.9,2);
\draw[red, dashed]  (3.9,3) -- (3,1.5);
\node at (3,1.5) {$\cdot$};
\node at (3.6,2.5) {$\cdot$};

\draw[violet, dashed] (3.7,2.5) --(2.8,1);
\draw[violet, dashed] (2.8,3) -- (2.5,2.5);
\node at (2.5,2.5) {$\cdot$};
\node at (3.1,1.5) {$\cdot$};

\draw[yellow, dashed]  (3.2,2.5) -- (2.3,1);
\draw[yellow, dashed] (2.3,2) -- (2,1.5);
\node at (2,1.5) {$\cdot$};
\node at (2.6, 1.5) {$\cdot$};

\draw[gray, dashed] (3.2,1.5) --(2.9,1);
\draw[gray, dashed] (2.9, 3) -- (2.5,2.35);
\node at (2.6,2.5) {$\cdot$};
\draw[gray, dashed] (4.5,2.35) -- (4.3,2);
\draw[gray, dashed] (4.3,4) -- (4,3.5);
\node at (4,3.5) {$\cdot$};

\draw[orange, dashed] (2.5,2) -- (1.3,0);
\draw[orange, dashed] (1.3,1) -- (1,0.5);
\node at (1,0.5) {$\cdot$};
\node at (1.6,0.5) {$\cdot$};

\draw[magenta, dashed] (1.7,1.5) -- (1.5,1.15);
\draw[magenta, dashed] (3.5,1.15) -- (3.4,1);
\draw[magenta, dashed] (3.4,3) -- ( 2.5,1.5);
\node at (2.5,1.5) {$\cdot$};
\node at (3.1,2.5) {$\cdot$};

\draw[blue, dashed] (1.7,0.5) -- (1.4,0);
\draw[blue, dashed] (1.4,1) -- (0.8,0);
\node at (1.1,0.5) {$\cdot$};
\draw[blue, dashed]  (0.8,1) -- (0.5,0.5);
\node at (0.5,0.5) {$\cdot$};

\draw[purple, dashed] (1.2,0.5) -- (0.9,0);
\draw[purple, dashed] (0.9,1) -- (0.3,0);
\node at (0.6,0.5) {$\cdot$};
\draw[purple, dashed]  (0.3,1) -- (0,0.5);
\node at (0,0.5) {$\cdot$};
\end{tikzpicture}
\caption{Partition of $X^{N_i}$ into sets of continuity of $(T^{\w^i,\theta})^{\ell_i}$.}\label{figb}
\end{figure}
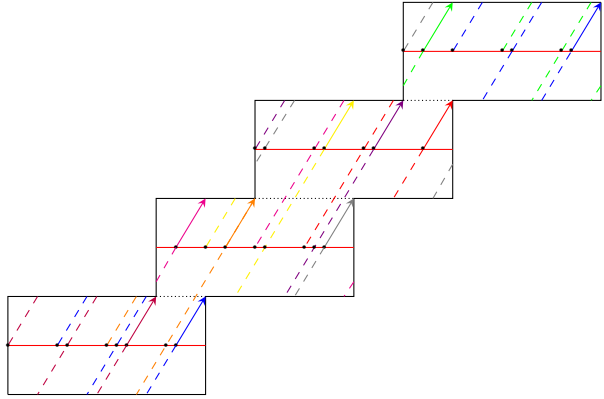

To define $\zeta^+$ for a direction $\theta \in B_i$ which has no saddle connection shorter than $2\ell_i$, we develop the argumentation when $\theta \in \SS1$
 is in the interior of the first quadrant, and leave to the reader to derive the corresponding formulas for $\theta$ in the other quadrants.
   Consider the finite collection of  points
   $$\Sigma^{\w^i,\theta,N_i} := \big \{(k,j - 1/2\tan(\theta)):  j \in \{\w^i_{k-1},2-\w^i_k,2\} \big \} \cap X^{N_i}$$
(see Figure \ref{figa}). $\Sigma^{\w^i,\theta,N_i}$ consists of  the points inside the ring for which the trajectory hits a singularity before hitting the next cross-section.

 All the points in $\Sigma^{\w^i,\theta,N_i}$ are distinct except for the pair of points
 $\{N_i\}\times\{2 -1/2\tan(\theta)\}$ and  $\{N_i\}\times\{2 - \w^i_{N_i} - 1/2\tan(\theta)\}$ which are the same point since $\w^i_{N_i} = 0$,
as well as the pair of points
  $\{-N_i+1\}\times\{2 -1/2\tan(\theta)\}$ and  $\{-N_i + 1\}\times\{2 - \w^i_{-N_i} - 1/2\tan(\theta)\}$
 which are the same point since $\w^i_{-N_i} = 0$.
We call such a point $z$ a \textit{blocking point}. Here, and at the discussion below, it's worth to recall that for each integer $k$, $\{k\}\times[0,2)$ is a circle and we will always consider the coordinates of the circle modulo 2.
 
 Consider the collection of  points 
 $$\hat \Sigma^{\w^i,\theta,N_i} := \{(T^{\w^i,\theta})^{-j}z : j \in \{0,1,\dots,\ell_i\} \text{  and  }  z\in \Sigma^{\w^i,\theta,N_i} \}.$$
{Any point in this set corresponds to a unique $j \in \{0,\dots,\ell_i\}$ and a unique $z \in \Sigma^{\w^i,\theta,N_i}$.} This set
 partitions $X^{N_i}$ into a finite number of intervals, $ \{I_j^{\w^i,\theta}\}$, such that on each of these
intervals the forward  map $(T^{\w^i,\theta})^{\ell_i}$ is continuous (see Figure \ref{figb}).
 This partition can be defined more generally than for $\theta \in B_i$, it makes sense as soon as 
as there are no saddle connections of length at most $2 \ell_i$.
Thus for $\theta \in B_i$ and a small enough  neighborhood  $U(\theta)$, each interval $I^{\w^i,\theta'}_j$ varies continuously with respect to $\theta' \in U(\theta)$.

Let $\eta_i(\theta) > 0$ be the minimum of the lengths of the intervals $I^{\w^i,\theta'}_j$.
Let  
$$ C_i := \bigcup_{\theta \in B_i} U(\theta).$$
We naturally extend the definition of the function $\eta_i(\cdot)$, originally defined on $B_i$, to the larger 
set $C_i$, it is a continuous function of $\theta$. We additionally suppose that $U(\theta)$ is 
sufficiently small so that $\eta(\theta') > \eta_i(\theta)/2$ for all
$\theta' \in U(\theta)$. By construction $C_i$ is an open set containing $B_i$, thus $\lambda(C_i^c)< \gamma_i$.
Furthermore, starting with \eqref{e}, by continuity in $\theta$ we can suppose that for all $\theta \in C_i$
we have
\begin{equation}
\Big |H^{\w^i,\theta}_{j,n,\ell}(z) -\frac{\int_{X} h_j(y) \, d\mu(y)}{\int_{X} h_n(y) \, d\mu(y)} \Big | < 2\gamma_i.\label{eeeest}
\end{equation}
for all $\ell \ge \hat \ell$ and all  $z \in  X^{N_i}$ (except those $z$ whose orbit hits a singular point before time $\ell$).

We can do the same for backward orbits, we will choose $\ell_i$ to work for forward and backward orbits.  We will denote by $C_i$ the open set of directions satisfying both the
forward and backward conditions.

For $\eta_i > 0$ consider the open set
$$D_i := \{\theta \in C_i : \eta_i(\theta) > \eta_i\}.$$

Since the function $\eta_i(\theta)$ is strictly positive for every $\theta \in C_i$, we can choose 
 a small positive number $\eta_i$ such that the set 
 $\eta_i > 0$ such that the Lebesgue measure of the 
 set $D_i$  is at least $1-\gamma_i$.
So for all $\theta \in D_i$ the lengths of all the intervals $ \{I_j^{\w^i,\theta} \}$
 are at least $\eta_i > 0$.

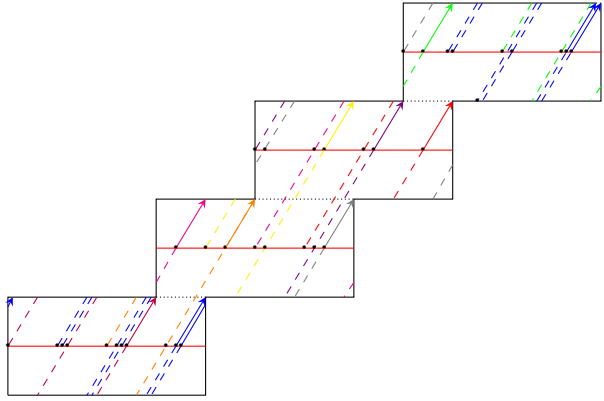
\begin{figure}[h]
\begin{tikzpicture}[scale=1.3]

 \draw[] (0,0) -- (2,0) -- (2,1)-- (3.5,1)--(3.5,2) --(4.5,2) -- (4.5,3) -- (6,3) -- (6,4);
  \draw[] (0,0) -- (0,1) -- (1.5,1)-- (1.5,2)-- (2.5,2) -- (2.5,3) -- (4,3) -- (4,4) -- (5.95,4); 
   \draw[densely dotted] (5.95,4) -- (6,4);
\draw[densely dotted] (1.5,1) -- (2,1);
\draw[densely dotted] (2.5,2) -- (3.5,2);
\draw[densely dotted] (4,3) -- (4.5,3) ;
\draw[red] (0,0.5) -- (2,0.5);
\draw[red] (1.5,1.5) -- (3.5,1.5);
\draw[red] (2.5,2.5) -- (4.5,2.5);
\draw[red] (4,3.5) -- (6,3.5);

\draw[blue, ->, >=stealth] (1.7,0.5) -- (2,1);
\draw[purple, ->, >=stealth] (1.2,0.5) -- (1.5,1);
\draw[gray, ->, >=stealth] (3.2,1.5) -- (3.5,2);
\draw[orange, ->, >=stealth] (2.2,1.5) -- (2.5,2);
\draw[magenta, ->, >=stealth] (1.7,1.5) -- (2,2);
\draw[red, ->, >=stealth] (4.2,2.5) -- (4.5,3);
\draw[violet, ->, >=stealth] (3.7,2.5) -- (4,3);
\draw[yellow, ->, >=stealth] (3.2,2.5) -- (3.5,3);
\draw[green, ->, >=stealth] (4.2,3.5) -- (4.5,4);
\draw[blue, ->, >=stealth] (5.7,3.5) -- (6,4);

\node at (1.7,0.5) {$\cdot$};
\node at (1.2,0.5) { $\cdot$};
\node at (3.2,1.5) { $\cdot$};
\node at (2.2,1.5) { $\cdot$};
\node at (1.7,1.5) { $\cdot$};
\node at (4.2,2.5) { $\cdot$};
\node at (3.7,2.5) { $\cdot$};
\node at (3.2,2.5) { $\cdot$};
\node at (5.7,3.5) { $\cdot$};
\node at (4.2,3.5) { $\cdot$};

\draw[blue, dashed] (5.4,3) -- (5.7, 3.5);
\draw[blue, dashed] (5.4,4) -- (5.1, 3.5);
\node at (5.1,3.5) {$\cdot$};
\draw[blue, dashed] (5.1,3.5) -- (4.8,3);
\draw[blue, dashed] (4.8,4) -- (4.5,3.5);
\node at (4.5,3.5) {$\cdot$};

\draw[green, dashed] (4.2,3.5) -- (4,3.15);
\draw[green, dashed] (6,3.15) -- (5.9,3);
\draw[green, dashed] (5.9,4) -- (5.3,3);
\draw[green, dashed] (5.3,4) -- (5,3.5);
\node at (5,3.5) {$\cdot$};
\node at (5.6,3.5) {$\cdot$};

\draw[red, dashed] (4.2,2.5) -- (3.9,2);
\draw[red, dashed]  (3.9,3) -- (3,1.5);
\node at (3,1.5) {$\cdot$};
\node at (3.6,2.5) {$\cdot$};

\draw[violet, dashed] (3.7,2.5) --(2.8,1);
\draw[violet, dashed] (2.8,3) -- (2.5,2.5);
\node at (2.5,2.5) {$\cdot$};
\node at (3.1,1.5) {$\cdot$};

\draw[yellow, dashed]  (3.2,2.5) -- (2.3,1);
\draw[yellow, dashed] (2.3,2) -- (2,1.5);
\node at (2,1.5) {$\cdot$};
\node at (2.6, 1.5) {$\cdot$};

\draw[gray, dashed] (3.2,1.5) --(2.9,1);
\draw[gray, dashed] (2.9, 3) -- (2.5,2.35);
\node at (2.6,2.5) {$\cdot$};
\draw[gray, dashed] (4.5,2.35) -- (4.3,2);
\draw[gray, dashed] (4.3,4) -- (4,3.5);
\node at (4,3.5) {$\cdot$};

\draw[orange, dashed] (2.5,2) -- (1.3,0);
\draw[orange, dashed] (1.3,1) -- (1,0.5);
\node at (1,0.5) {$\cdot$};
\node at (1.6,0.5) {$\cdot$};

\draw[magenta, dashed] (1.7,1.5) -- (1.5,1.15);
\draw[magenta, dashed] (3.5,1.15) -- (3.4,1);
\draw[magenta, dashed] (3.4,3) -- ( 2.5,1.5);
\node at (2.5,1.5) {$\cdot$};
\node at (3.1,2.5) {$\cdot$};

\draw[blue, dashed] (1.7,0.5) -- (1.4,0);
\draw[blue, dashed] (1.4,1) -- (0.8,0);
\node at (1.1,0.5) {$\cdot$};
\draw[blue, dashed]  (0.8,1) -- (0.5,0.5);
\node at (0.5,0.5) {$\cdot$};

\draw[purple, dashed] (1.2,0.5) -- (0.9,0);
\draw[purple, dashed] (0.9,1) -- (0.3,0);
\node at (0.6,0.5) {$\cdot$};
\draw[purple, dashed]  (0.3,1) -- (0,0.5);
\node at (0,0.5) {$\cdot$};

\draw[blue, ->, >=stealth] (5.65,3.5) -- (5.95,4);
\node at (5.65,3.5)  {$\cdot$};

\draw[blue, dashed] (5.35,3) -- (5.65, 3.5);
\draw[blue, dashed] (5.35,4) -- (4.75,3);
\node at (4.75,3)  {$\cdot$};
\draw[blue, dashed] (4.75,4) -- (4.45,3.5);
\node at (4.45,3.5) {$\cdot$};

\draw[blue] (1.75,0.5) -- (2,0.925);

\draw[blue, ->, >=stealth] (0.,0.9) -- (0.05,1);
\node at (1.75,0.5) {$\cdot$};

\draw[blue, dashed] (1.75,0.5) -- (1.45,0);
\draw[blue, dashed] (1.45,1) -- (0.85,0);

\node at (1.15,0.5) {$\cdot$};
\draw[blue, dashed]  (0.85,1) -- (0.55,0.5);
\node at (.55,0.5) {$\cdot$};
\end{tikzpicture}
\caption{The partition changes continuously, the blocking points bifurcate.}\label{figbb}
\end{figure}

Now we consider a configuration $\w$, not necessarily ringed, but close to a ringed configuration. 
Assume now  that $\theta \in D_i$ and note
that the point $(T^{\w,\theta})^{-j}z$ varies continuously with respect to $\w$ close to $\w^i$, for any $z$ which 
is not a blocking point for $\w^i$.

If $w$ is very close to $w_i$ such that $\w_{\pm N_i} \ne 0$, then each of the two blocking points, $\{N_i\}\times\{2 -1/2\tan(\theta)\}$ and $\{-N_i+1\}\times\{2 -1/2\tan(\theta)\}$,   bifurcates
into a pair of points for $S_{\w}$.
The bifurcated points create two new intervals $\tilde I_+^{\w,\theta}= \{N_i\} \times (2 - \w_{N_i} -1/2\tan(\theta) ,2  - 1/2\tan(\theta))$
and $\tilde I_-^{\w,\theta}= \{-N_i + 1\} \times (2 -1/2\tan(\theta) ,2 + \w_{-N_i} - 1/2\tan(\theta))$
 in the partition induced by the set $\Sigma^{w,\theta}$  (see Figure \ref{figbb}).  
 
 Then, if we consider 
the partition $\{I_j^{w,\theta}\}$ defined in a similar way as the partition $\{I_j^{\w^i,\theta}\}$, we have  
at most  $2(\ell_i+1)$ new intervals $\{(T^{\w,\theta})^{-j} \tilde I_\pm^{\w,\theta}: j \in \{0,\dots,\ell_i\}\}$. 
We call this collection of  new intervals $\mathcal{I}^{w,\theta,N_i,*}$. 

Let $V^{\w,\theta, N_i} := \cup_{I\in \mathcal{I}^{w,\theta,N_i,*}}I$ be the union of all the intervals in the 
collection $\{I_j^{w,\theta}\}$. Since $\{I_j^{\w^i,\theta}\}$ is in bijection with a subset of $\{I_j^{\w,\theta}\}$
(except for at most $2(\ell_i+1)$ new intervals) we can define an injective, piecewise affine transformation $\zeta^+ = \zeta^{+,\w^i,\w,n,\theta}$ 
from  $V^{\w^i,\theta, N_i}$ to $V^{\w,\theta, N_i}$ that sends  every $I_j^{\w^i,\theta}$ affinely to the
corresponding $I_j^{\w,\theta}$. Note that, for $\w$ sufficiently close to $\w^i$, this piecewise affine transformation is 
close to the identity.

It is important to underline that the definition of the map $\zeta^+$ guarantees that the points 
$z \in X\setminus V^{\w,\theta,N_i}$ 
for which the $T^{\w^i, \theta}$ forward orbit of $\zeta^+(z)$ hits a singular point before time $\ell_i$ are points
for which the $T^{\w, \theta}$ forward orbit of $z$ hits a singular point before time $\ell_i$.
Furthermore note that $X^{N_i} \setminus V^{\w,\theta,N_i}$ is the set of points where $\zeta^+$ is not defined. It consists of those points 
whose $T^{\w,\theta}$ forward orbit segment of length $\ell_i$ exits the ring. In Figure \ref{figbb} these are the points
in between the close blue orbit segments which have bifurcated from a blocking orbit segment.
 These forward orbit segments  have no usable relation to  $T^{\w^i,\theta}$ forward orbit segments of length $\ell_i$.

Fix $\theta \in \overline{D_i}$ and suppose that $\w$ is  close enough to $\w^i$
so that the above bijection is defined. Then we 
can define   $\eta_i(\w,\theta)$ to be the infimum length of all the intervals $I_j^{\w,\theta}$. 
For $\w$ sufficiently close to $\w^i$, for each $\theta \in C_i$, the function $\eta_i(\w,\theta)$ varies continuously with $\w$ since the set $\overline{C}_i$ does 
not contain any saddle connection directions of length at most $2 \ell_i$. Let 
 $\eta_i(\w) := \inf_{\theta \in \overline{D_i}} \eta_i(\w,\theta) \ge 0$.
 We have that $\eta_i(\w^i) = \eta_i(\w^i,\theta_0)$ for some $\theta_0 \in \overline{D_i}$ and
 thus $\eta_i(\w)$  is locally a continuous function of $\w$.

By the definition of $D_i$ we have $ \eta_i(\w^i) \ge \eta_i$, we require that 
\begin{itemize}
\item[\bf A1] 

$\delta_i$ is so small that $\eta_i(\w)$ varies continuously with $\w \in \mathcal{U}_{\delta_i}(\w^i)$ and
$\eta_i(\w) > \eta_i/2$ for each $\w \in \mathcal{U}_{\delta_i}(\w^i)$.

\noindent
We make the analogous assumption for the ordered sets arising from backward orbits.
\end{itemize}

 We require a further closeness condition on those orbits which we can control. 
\begin{itemize}
\item[\bf A2] We  assume that $\delta_i > 0 $ is so small that for all $\w \in \U_{\delta_i}(\w^i)$,
for all   $z \in X^{N_i}$ for which $\zeta^+$ is defined, for all $\theta\in D_i$, we have
 $$\big |H^{\w,\theta}_{j,n,\ell_i} (z) - H^{\w^i,\theta}_{j,n,\ell_i}(\zeta^+(z))\big |<\gamma_i.$$ 
 Again we make the analogous assumption for $\zeta^-$.
 \end{itemize}

In the proof we will only consider directions in $\overline{D_i}$, the complement are  directions where we are not able to control the Hopf averages.  As already mentioned, we also are not able to control the Hopf averages of the forward orbits of points 
in $X^{N_i} \setminus V^{\w,\theta,N_i}$.
Our main goal is to  control the backwards Hopf averages of such orbits.

A  key point of the proof is 
that if   $\theta \in D_i$, then  for each $z \in X^{N_i}$ 
at least one of  $\zeta^+(z)$ or  $\zeta^-(z)$ is defined.  To do this
consider any $\theta \in D_i$ and the finite set  $\hat{\Sigma}^+ := \hat{\Sigma}^{\w^i,\theta,N_i}$ and the analogous
set $\hat{\Sigma}^-$.  These two sets do not intersect, for if they did there would be a saddle connection of length at most $2\ell_i$.
Let $\iota_i(\theta) > 0 $ be the minimum of the distances between pairs of 
points in these two collections and $\iota_i := \inf_
{\theta \in \overline{D_i}} \iota_i(\theta)$.  Since $C_i$ avoids a neighborhood of the directions of the saddle connections of length at most $2\ell_i$ we
have $\iota_i > 0$.

We define $\iota_i(\w)$ in an analogous way.

\begin{itemize}
	\item[\bf A3] We assume that $\delta_i > 0 $ is so small that for each $\w \in U_{\delta_i} (\w^i)$ $\iota_i(\w) > \iota_i/2$.
\end{itemize}

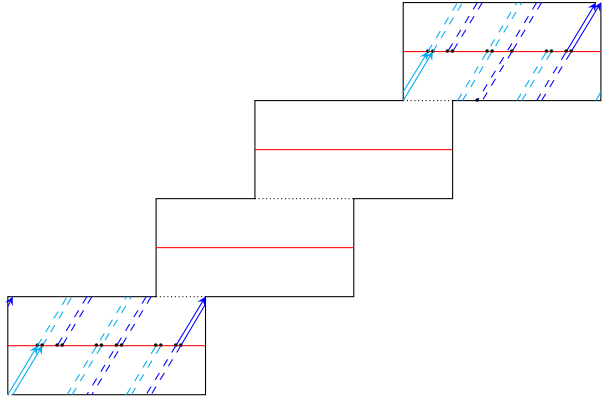
\begin{figure}[h]
\begin{tikzpicture}[scale=1.3]

 \draw[] (0.05,0) -- (2,0) -- (2,1)-- (3.5,1)--(3.5,2) --(4.5,2) -- (4.5,3) -- (6,3) -- (6,4);
 \draw[] (0,0) -- (0,1) -- (1.5,1)-- (1.5,2)-- (2.5,2) -- (2.5,3) -- (4,3) -- (4,4) -- (5.95,4);

  \draw[densely dotted] (0,0) -- (0.05,0);
 
\draw[densely dotted] (1.5,1) -- (2,1);
\draw[densely dotted] (2.5,2) -- (3.5,2);
\draw[densely dotted] (4,3) -- (4.5,3) ;
\draw[red] (0,0.5) -- (2,0.5);
\draw[red] (1.5,1.5) -- (3.5,1.5);
\draw[red] (2.5,2.5) -- (4.5,2.5);
\draw[red] (4,3.5) -- (6,3.5);

\draw[blue, ->, >=stealth] (1.7,0.5) -- (2,1);
\draw[blue, ->, >=stealth] (5.7,3.5) -- (6,4);
\node at (1.7,0.5) {$\cdot$};
\node at (5.7,3.5) { $\cdot$};
\draw[blue, dashed] (5.4,3) -- (5.7, 3.5);
\draw[blue, dashed] (5.4,4) -- (5.1, 3.5);
\node at (5.1,3.5) {$\cdot$};
\draw[blue, dashed] (5.1,3.5) -- (4.8,3);
\draw[blue, dashed] (4.8,4) -- (4.5,3.5);
\node at (4.5,3.5) {$\cdot$};
\draw[blue, dashed] (1.7,0.5) -- (1.4,0);
\draw[blue, dashed] (1.4,1) -- (0.8,0);
\node at (1.1,0.5) {$\cdot$};
\draw[blue, dashed]  (0.8,1) -- (0.5,0.5);
\node at (0.5,0.5) {$\cdot$};

\draw[cyan, ->, >=stealth]   (0,0) -- (0.3,0.5);
\node at (0.3,0.5) {$\cdot$};
\draw[cyan, dashed] (0.3,0.5) -- (0.6,1);
\draw[cyan, dashed] (0.6,0) -- (1.2,1);
\node at (0.9,0.5) {$\cdot$};
\draw[cyan, dashed] (1.2,0) -- (1.5,0.5);
\node at (1.5,0.5) {$\cdot$};

\draw[cyan, ->, >=stealth]    (4,3)--(4.3,3.5);
\node at (4.3,3.5) {$\cdot$};
\draw[cyan, dashed] (4.3,3.5) -- (4.6,4);
\draw[cyan, dashed] (4.6,3) -- (5.2,4);
\node at (4.9,3.5) {$\cdot$};
\draw[cyan, dashed] (5.2,3) -- (5.5,3.5);
\node at (5.5, 3.5) {$\cdot$};

\draw[blue, ->, >=stealth] (5.65,3.5) -- (5.95,4);
\node at (5.65,3.5)  {$\cdot$};
\draw[blue, dashed] (5.35,3) -- (5.65, 3.5);
\draw[blue, dashed] (5.35,4) -- (4.75,3);
\node at (4.75,3)  {$\cdot$};
\draw[blue, dashed] (4.75,4) -- (4.45,3.5);
\node at (4.45,3.5) {$\cdot$};

\draw[blue] (1.75,0.5) -- (2,0.925);
\draw[blue, ->, >=stealth] (0.,0.9) -- (0.05,1);
\node at (1.75,0.5) {$\cdot$};

\draw[blue, dashed] (1.75,0.5) -- (1.45,0);
\draw[blue, dashed] (1.45,1) -- (0.85,0);

\node at (1.15,0.5) {$\cdot$};
\draw[blue, dashed]  (0.85,1) -- (0.55,0.5);
\node at (0.55,0.5) {$\cdot$};

\draw[cyan, ->, >=stealth]   (0.05,0) --  (0.35,0.5) ;
\node at (0.95,0.5) {$\cdot$};
\draw[cyan, dashed] (0.35,0.5) -- (0.65,1);
\draw[cyan, dashed] (0.65,0) -- (1.25,1);
\node at (0.35,0.5) {$\cdot$};
\draw[cyan, dashed] (1.25,0) -- (1.55,0.5);
\node at (1.55,0.5) {$\cdot$};

\draw[cyan, ->, >=stealth]   (4,3.09) --  (4.25,3.5) ;
\node at (4.25,3.5) {$\cdot$};
\draw[cyan]    (5.95, 3) -- (6,3.09);
\draw[cyan, dashed] (4.25,3.5) -- (4.55,4);
\draw[cyan, dashed] (4.55,3) -- (5.15,4);
\node at (4.85,3.5) {$\cdot$};
\draw[cyan, dashed] (5.15,3) -- (5.45,3.5);
\node at (5.45, 3.5) {$\cdot$};

\end{tikzpicture}
\caption{Uncontrolled forward orbits do not intersect uncontrolled backward orbits.}\label{figc}
\end{figure}

Fix $i$, by the triangle inequality, for any $z \in X^{N_i}$ such that $\theta \in D_i$ and $\zeta^+(z)$ is defined,   we have

$$\begin{array}{ccc}
\Big |H^{\w,\theta}_{j,n,\ell_i} (z) - \frac{\int_{X} h_j(y,\psi') \, d\mu}{\int_{X} h_n(y,\psi') \, d\mu} \Big | & \hspace{-0.2cm} \le & \hspace{-0.4cm}  \Big |H^{\w,\theta}_{n,\ell_i} h_j(z) - H^{\w^i,\theta}_{n,\ell_i}h_j(\zeta^+(z))\Big | \\
&&  \hspace{0.2cm} + \Big |H^{\w^i,\theta}_{n,\ell_i} h_j(\zeta^+(z)) -  \frac{\int_{X} h_j(y,\theta) \, d\mu}{\int_{X} h_n(y,\theta) \, d\mu} \Big |. \\
\end{array}
$$
By {\bf A2}, the first term is smaller than $\gamma_i$. 
Now we consider the second term.

By assumption $\theta \in D_i$ and thus  
if the orbit of $\zeta^+(z)$ does not hits a singular point before time $\ell$, then \eqref{eeeest} implies

$$\Big |H^{\w^i,\theta}_{n,\ell_i} h_j(\zeta^+(z)) -   \frac{\int_{X} h_j(y,\theta) \, d\mu}{\int_{X} h_n(y,\theta) \, d\mu}  \Big | < 2 \gamma_i$$
for all $1 \le j\le i$, and $1 \le n \le N_i$. Combining with {\bf A2} we conclude
$$\Big |H^{\w,\theta}_{n,\ell_i} h_j(z) - \frac{\int_{X} h_j(y,\psi') \, d\mu}{\int_{X} h_n(y,\psi') \, d\mu} \Big | < 3 \gamma_i$$
for all $z$ where $\zeta^+$ is defined,  with $\theta \in D_i$, for all $(j,n) \in \J_i$, for all $\w \in \U_{\delta_i}(\w^i)$
as long as the $T^{\w^i,\theta}$ forward orbit of $\zeta^+(z)$ of length $\ell_i$ does not hit a singular point.
As already pointed out we can thus consider points  $z \in X^{N_i}$ such that $\zeta^+$ is defined an the 
the $T^{\w,\theta}$ forward orbit of $z$ of length $\ell_i$ does not hit a singular point.

We make the analogous estimate on the past  
$$\Big |H^{\w,\theta}_{n,-\ell_i} h_j(z) - \frac{\int_{X} h_j(y,\theta) \, d\mu}{\int_{X} h_n(y,\theta) \, d\mu} \Big | < 3 \gamma_i$$
for all $z$ where $\zeta^-$ is defined,  with $\theta \in D_i$, for all $(j,n) \in \J_i$, for all $g \in \U_{\delta_i}(\w^i)$
as long as the backwards orbit of $\zeta^-(z)$ of length $\ell_i$ is defined.

Using {\bf A3}, since $\lambda(D_i) > 1 - 2 \gamma_i$, the set 
$\mathcal{H} = \cap_{M=1}^\infty \cup_{i=M}^\infty D_i$ has full measure. 
Fix $\w \in \G$ and  $\theta \in \mathcal{H}$, then there is an infinite sequence $i_k$ such that $\w \in \U_{\delta_{i_k}}(\w^{i_k})$ and $\theta \in \Theta_{i_k}$ for all $k$.  
 Since  $J_{i_k} \to \infty$
we can conclude that 
for all $z \in X$ with $\theta \in \mathcal{H}$, for each $j \ge 1$, for each $n \ge 1$  either
\begin{equation}\label{e11}
\lim_{k \to \infty} H^{\w,\theta}_{j,n,\ell_{i_k}}(z) = \frac{\int_{X} h_j(y,\theta) \, d\mu}{\int_{X} h_n(y,\theta) \, d\mu}.
\end{equation}
or
\begin{equation}\label{e12}
\lim_{k \to \infty} H^{\w,\theta}_{j,n,-\ell_{i_k}}(z) =  \frac{\int_{X} h_j(y,\theta) \, d\mu}{\int_{X} h_n(y,\theta) \, d\mu}.
\end{equation}

Lemma \ref{lemma6}  finishes the proof of the theorem.
\end{proofof}

\section{Wind-tree proof}\label{secWT}
The proof of result in the wind-tree setting is essentially the same, there are three differences which we will 
point out. The first difference is the cross-section to the flow which we use.  In the wind-tree setting we will use
the first return map to the boundaries of the rhombi, and we will show that the structure of this cross-section is
quite similar to the cross-section we used for staircases.  The second difference is the description of ringed-configurations.
The final difference is the description of the set corresponding to $\Sigma_{\w^i,\theta,N_i}$ in Figure \ref{figa}.  In particular we have
many more blocking points, and they can bifurcate in a slightly  more complex way. Thus the domain of definition of the maps
$\zeta^{\pm}$ are more complicated.

\subsection{The phase space}
If a rhombus does not intersect any other rhombus at each corner there are three directions pointing to the interior of the table, 
while for all other points there are only two such directions  (see Figure \ref{fig3}).  Intersecting rhombi can have slightly different behavior, but we do not need to describe it since they will only occur in the proof in a very special way.

\begin{figure}[h]

\begin{minipage}[ht]{0.45\linewidth}
\centering
\begin{tikzpicture}[rotate=45, scale=1]

\draw[] (0,0) rectangle +(2,2);

\draw [](1,2) edge[->]  (1+1.732/4,2+1/4) ;
\draw [](1,2) edge[->]  (1-1.732/4,2+1/4) ;

\draw [](1,0) edge[->]  (1+1.732/4,-1/4) ;
\draw [](1,0) edge[->]  (1-1.732/4,-1/4) ;

\draw [](0,1) edge[->]  (-1.732/4,1+1/4) ;
\draw [](0,1) edge[->]  (-1.732/4,1-1/4) ;

\draw [](2,1) edge[->]  (2+1.732/4,1+1/4) ;
\draw [](2,1) edge[->]  (2+1.732/4,1-1/4) ;

\draw [](2,2) edge[->]  (2-1.732/4,2+1/4) ;
\draw [](2,2) edge[->]  (2+1.732/4,2+1/4) ;
\draw [](2,2) edge[->]  (2+1.732/4,2-1/4) ;

\draw [](0,2) edge[->]  (-1.732/4,2+1/4) ;
\draw [](0,2) edge[->]  (1.732/4,2+1/4) ;
\draw [](0,2) edge[->]  (-1.732/4,2-1/4) ;

\draw [](0,0) edge[->]  (-1.732/4,1/4) ;
\draw [](0,0) edge[->]  (1.732/4,-1/4) ;
\draw [](0,0) edge[->]  (-1.732/4,-1/4) ;

\draw [](2,0) edge[->]  (2-1.732/4,-1/4) ;
\draw [](2,0) edge[->]  (2+1.732/4,1/4) ;
\draw [](2,0) edge[->]  (2+1.732/4,-1/4) ;
\end{tikzpicture}
\caption{The phase space of one rhombus.}\label{fig3}
\end{minipage}\nolinebreak
\begin{minipage}[ht]{0.55\linewidth}
\centering
\begin{tikzpicture}[rotate=45,scale=0.5]

\draw[red,directed] (0,2.3) -- (0,4.3);
\draw[red,directed] (0,4.3) -- (2,4.3);
\draw [red](1,4.3) edge[->]  (1-1.732*0.5,4.3+0.5) ;
\draw [red](0,3.3) edge[->]  (-1.732*0.5,3.3+0.5) ;
\draw [red,dotted] (0,2.3) -- (2, 4.3) ;

\draw[blue,directed] (4.3,2) -- (4.3,0);
\draw[blue,directed] (4.3,0) -- (2.3,0);
\draw [blue](3.3,0) edge[->]  (3.3+1.732*0.5,-0.5) ;
\draw [blue](4.3,1) edge[->]  (4.3+1.732*0.5,1-0.5) ;
\draw [blue,dotted] (4.3,2) -- (2.3, 0) ;

\draw[brown,directed] (0,2) -- (0,0);
\draw[brown,directed] (0,0) -- (2,0);
\draw [brown](1,0) edge[->]  (1-1.732*0.5,-0.5) ;
\draw [brown](0,1) edge[->]  (-1.732*0.5,1-0.5) ;
\draw [brown,dotted] (0,2) -- (2,0) ;

\draw[directed] (4.3,2.3) -- (4.3,4.3);
\draw[directed] (4.3,4.3) -- (2.3,4.3);
\draw [](3.3,4.3) edge[->]  (3.3+1.732*0.5,4.3+0.5) ;
\draw [](4.3,3.3) edge[->]  (4.3+1.732*0.5,3.3+0.5) ;
\draw [dotted] (4.3,2.3) -- (2.3,4.3) ;

\end{tikzpicture}
\caption{The phase space is the disjoint union of four closed oriented ``intervals''.}\label{fig4}
\end{minipage}

\end{figure}

By taking the arc length coordinate $\s$ along the diagonal of a rhombus, 
we think of  the contribution of the rhombus with center 
$z_n$ to $X^{g,\theta}$ as the union of four closed 
intervals $I^n_\phi$ indexed by 
$\phi \in [\theta] := \{\theta, \frac32\pi - \theta, \pi + \theta, \frac12\pi-\theta \}$, each of these
intervals corresponds to the cartesian product of  the two intersecting sides  of the rhombus with a fixed inner pointing direction  $\phi$ (see Figure \ref{fig4}) (as before the word inner means pointing into the table, so  away from the rhombus).   Since $\theta$ is assumed not parallel to a side,  the set $[\theta]$ is in bijection with the set $\{1,2,3,4\}$ by choosing the  quadrant containing  $\phi \in [\theta]$.

Note that for any $n$ and $\phi$ the interval $I^n_\phi$ is just translated and rotated copy of the interval $I = [0,s]$.  Using this  we can drop the index $n$ from the notation.
This yields to the following identification.
\begin{equation}\label{phasespace}
X^{\theta} = \N \times I \times \{1,2,3,4\}.
\end{equation}

Let
$$\U_\varepsilon(g) := \{g'  \in \textit{Conf} : d_H(g',g)<\varepsilon\}.$$ 

\begin{proposition}\label{prop1}
There is a dense $G_{\delta}$ subset $G$ of $(\textit{Conf},d_H)$ such that for each $g \in G$ 
\begin{enumerate}
\item $g$ is an infinite configuration,
\item every pair of points $z_1,z_2 \in g$ satisfy $d(z_1,z_2) > s$.
\item there is an enumeration of  $g$  such that $d(z_n,0)$ is a strictly increasing function.
\end{enumerate}
\end{proposition}

Part (3) of the Proposition allows us to assert that the identification of the phase space with the set $X^{\theta}$ varies continuously
with the parameter $g \in G$.  The proof uses parameters which are not in $G$, however we will only use the identification
with a compact part of $X^{\theta}$ which will vary continuously with the parameter.

\begin{proof}
Points (1) and (2) were proven in \cite{MSTr4}. We start with a dense set of configurations $\{g_i\}$, such that each configuration is finite, 
and moreover $g_i = \{z^{g_i}_1, \dots , z^{g_i}_i\}$ satisfies condition (3).
Let $\delta_i > 0$ be so small
that $d(z^{g'}_n,0)$ is strictly increasing for $n=1,\dots, i$
for  any $g' \in \U_{\delta_i} (g_i)$. Point (3) follows immediately for the dense $G_\delta$ set
$\cap_{j=1}^{\infty} \cup_{i \ge j} \U_{\delta_i}(g_i)$.
\end{proof}

\begin{figure}[t]
\begin{minipage}[ht]{0.5\linewidth}
\centering
\begin{tikzpicture}[rotate=45,scale=0.75]

\draw[thin] (0.1,0.2) rectangle +(0.5,0.5);
\draw[thin] (0.3,-0.7) rectangle +(0.5,0.5);

\draw[thin] (-0.8,-0.6) rectangle +(0.5,0.5);
\draw[thin] (-1.5,0.4) rectangle +(0.5,0.5);
\draw[thin] (1,0.4) rectangle +(0.5,0.5);

\draw[thin,lightgray](-2,-2)--(-2,2)--(2,2)--(2,-2)--(-2,-2);
 
\foreach \i in {-2.25,-1.75,...,1.75}
\draw[](\i,1.75) rectangle +(0.5,0.5);
 
 \foreach \i in {-2.25,-1.75,...,1.75}
 \draw[](\i,-2.25) rectangle +(0.5,0.5);
  
 \foreach \i in {-1.75,-1.25,...,1.25}
 \draw[](1.75,\i) rectangle +(0.5,0.5);
 
 \foreach \i in {-1.75,-1.25,...,1.25}
 \draw[](-2.25,\i) rectangle +(0.5,0.5);
\end{tikzpicture}
\end{minipage}\nolinebreak
\begin{minipage}[ht]{0.5\linewidth}
\begin{tikzpicture}[rotate=45,scale=0.75]

 \draw[thin] (0.1,0.2) rectangle +(0.5,0.5);
 \draw[thin] (0.3,-0.7) rectangle +(0.5,0.5);

 \draw[thin] (-0.8,-0.6) rectangle +(0.5,0.5);
\draw[thin] (-1.5,0.4) rectangle +(0.5,0.5);
\draw[thin] (1.1,0.4) rectangle +(0.5,0.5);
 
 \foreach \i in {-2.4,-1.85,...,2.3}
 \draw[](\i,1.95) rectangle +(0.5,0.5);
 
 \foreach \i in {-2.4,-1.85,...,2.3}
  \draw[](\i,-2.45) rectangle +(0.5,0.5);
  
 \foreach \i in {-1.9,-1.35,...,1.55}
 \draw[](2,\i) rectangle +(0.5,0.5);
 
    \foreach \i in {-1.9,-1.35,...,1.55}
 \draw[](-2.4,\i) rectangle +(0.5,0.5);

\end{tikzpicture}
\end{minipage}\nolinebreak
\caption{An 8-ringed configuration and a configuration close to it.}\label{fig}
\end{figure}
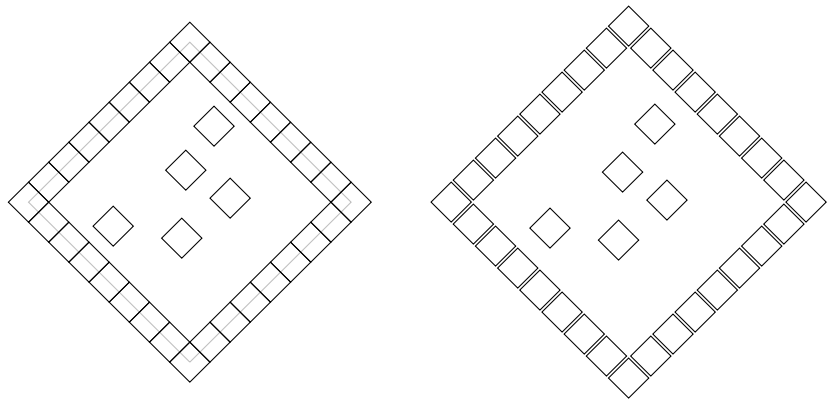

\subsection{Ringed configurations}
A configuration $h$ is \textit{$n$-ringed}  if the following three conditions hold:\\
\noindent
1) the boundary  $R(n):=\{(x,y) \in \R^2: |x| + |y| = ns\}$ of the corresponding rhombus is completely covered
by trees as in Figure \ref{fig} left (i.e., the obstacles covering the boundary intersect with each other on a whole side or do not intersect at all, and the boundary $R(n)$ intersection with each tree consists of two points which are in the middle of two sides of the tree),\\
\noindent
2)  the obstacles inside the ring, which we will call  \textit{interior obstacles},
can be enumerated $\{z_1,\dots,z_m\}$ in a way such that their distance to the origin is a strictly increasing function,\\
\noindent
3) no interior obstacle touches any other obstacle. 

We start with a dense set  of parameters  $\{f_i\}$ such that each $f_i$ is an $N_i$-ringed configuration, $N_i$ is increasing 
with $i$, and $f_i$ has at least $i$ interior obstacles. We define $X^{\theta,N_i} := \{1,\dots,m\} \times I \times \{1,2,3,4\}$ to be the part of the phase space corresponding to the inner obstacles; it plays the role of the set $X^{N_i}$ in the staircase case.

\subsection{The set corresponding to $\Sigma^{\w^i,\theta,N_i}$ and the maps $\zeta^{\pm}$} 
Consider the corners of the rhombi  which are inside the convex hull of $R(n)$; this includes all the corners of inner rhombi
and some corners of rhombi on the ring.  Unlike for the staircase,  the corners of inner obstacles are already on our cross-section, 
while for those points which are not in the set $X^{\theta,N_i}$, we flow them backward 
until they hit the cross-section $X^{\theta,N_i}$ for the first time, these are the \textit{blocking points}.
This finite collection of points plays the role of the set $\Sigma^{\w^i,\theta,N_i}$  for staircases.
To define $\zeta^+$ we proceed like for staircases, the 
   backward $(T^{f_i,\theta})^{\ell_i}$-orbits  of these points partition $X^{\theta,N_i}$ into a finite number of intervals $ \{I_j^{\w^i,\theta}\}$, on each such
interval the forward  map $(T^{f_i,\theta})^{\ell_i}$ is continuous. The blocking points bifurcate into two or three points, depending on if
two or three rhombi touch at the blocking point.  

\section{Appendix}
\subsection{Conservativity of $(T^{\w,\t},X,\mu)$} \label{sec:cons}

Let $T=T^{\w,\t}$. An \textit{(almost invariant) box sequence} is an increasing sequence 
$(Y_n)$ of open subsets of $X$ such that:
\begin{enumerate}
	\item $X=\cup Y_n$
	\item each $Y_n$ is of finite measure
	\item $\lim_{n\to\infty}\mu\left(T(Y_n)\setminus Y_n\right)=0$.
\end{enumerate}

\begin{lemma}[Boxes lemma]
If there is an almost invariant box sequence then $(T,X,\mu)$ is conservative.
\end{lemma}
\begin{proof}
	Let $(Y_n)$ be a box sequence for $T$.
	Let $U$ be a non-empty open set. We want to show that $U$ is \textit{non-wandering}, i.e., that there is an $i$ such that $T^i(U)\cap U\neq \emptyset$.
	
	Let $V_n=U\cap Y_n$. By the definition of a box sequence, we have that $U=\cup V_n$. Thus, for some $n_0$, $V_{n_0}$ is non-empty. Let $V=V_{n_0}$ and $n\ge n_0$ large enough so $\mu(T(Y_n)\setminus Y_n)<\mu(V)$.  Then we have $V\subset Y_n$. 
	
	For each $i\ge 1$, let $A^n_i$ be the set of points in $V$ who escape from $Y_n$ at time $i$, namely 
	$$A^n_i=\{x\in V : T^i(x)\not\in Y_n\text{ and }\forall 0\le k< i, T^k(x)\in Y_n\}.$$
	
Let us suppose the sets $T^i(A^n_i)$ are pairwise disjoint. Note that $\cup_iT^i(A^n_i)\subset T(Y_n)\setminus Y_n$, then we have $$\mu\left(\bigcup_iA^n_i\right)\le\Sigma_i\mu\left(A^n_i\right)=\Sigma_i\mu\left(T^i(A^n_i)\right)$$
$$=\mu\left(\bigcup_iT^i(A^n_i)\right)\le\mu(T(Y_n)\setminus Y_n))<\mu(V).$$
	Thus the set $F= V\setminus \cup_iA^n_i$, which is the set of points in $V$ whose (positive) iterates always stay 
in $Y_n$, is of positive measure. Given that $Y_n$ is of finite measure, the Poincar\'e recurrence Theorem applied to $\cup_{i \ge 0} T^i(F)$ implies that  this set $F$ is 
non-wandering. Thus $V$ and $U$ are also non-wandering.

On the other hand if the sets $T^i(A^n_i)$ are not pairwise disjoint, then there exists $i\neq j$ such that $T^i(A^n_i)\cap T^j(A^n_j)\neq \emptyset$. Note that $A^n_i$, $A^n_j$, $T^i(A^n_i)$ and $T^j(A^n_j)$ have all non-empty interiors. Thus it follows that $T^{i-j}(A^n_i)\cap A^n_j\neq \emptyset$. But $A^n_i$ and $A^n_j$ are both subsets of $V$. Thus $V$ (and also $U$) is non-wandering.
\end{proof}

\begin{proof} (Proposition \ref{prop:cons})
Consider a strictly increasing sequence $i_n$ and a strictly decreasing sequence $j_n$ such that 
$w_{i_k} \to 0$ and $w_{j_k} \to 0$.	Then the sequence of sets $X_n := \{j_k +1 , \dots, i_k\} \times [0,2)$ is a sequence
of boxes for all $\theta$.
\end{proof}

\subsection{Cantor representation}\label{secCantor}
Consider the  mapping $T: X \times \mathbb{S}^1 \to X \times \mathbb{S}^1$, induced by the first return map to the 
cross section $X$ of the translation flow in 
the direction $\theta$ on a compacted connected translation surface. 
For each $\theta$, the map $T|_{X \times \{\theta\}}$ is an interval exchange transformation (up to the behavior on the endpoints of the intervals).  
The following construction can be found in \cite{GaKrTr}.
Label the sets of continuity of $T$ by a finite enumerated alphabet $\mathcal{A}$, these are connected regions, 
with piecewise smooth boundary
(one can even choose coordinates so that they are polygonal regions \cite{Ka}).
Let $X^\infty$ be the set of $(x,\theta)$ whose forward orbit does not hit a singularity 
(nor an endpoint of the cross section if the cross section does not start and end at a singular point), 
and for $(x,\theta) \in X^{(\infty)}$  
let $i(x,\theta) \in \mathcal{A}^{{\mathbb{N}}}$  be the coding map defined by 
$i(x,\theta)_j = k \in \mathcal{A}$ if and only if $T^j(x,\theta)$ is in the $k$th element of $\mathcal{A}$. 
Let $K := \overline{ i(X^{(\infty)})} \subset  \mathcal{A}^{{\mathbb{N}}}$.
Let $\hat T$ denote the shift map on $K$, then by defnition $\hat T \circ i = i \circ T$ on $X^{(\infty)}$.  
Each point $(x,\theta) \in X \times \mathbb{S}^1 \setminus X^\infty$
has exactly two images in $K$. The main result of \cite{GaKrTr} it that the inverse $i^{-1}: K \to X \times \SS1$   satisfies
that the set $i^{-1} (s)$ is in fact a single point except when $s$ is a periodic point. If $s$ is a periodic point then $i^{-1}(s)$
is a periodic cylinder in $X$.

Finally let $K^\theta := i^{-1} (X \times \{\theta\}$), this is the Cantor representation of the corresponding IET.

\subsection{Uniform convergence of Hopf averages}\label{secUniform}
Both in Sections \ref{secStair} and \ref{secWT},  the first return maps  we consider, $T^{\w^i,\theta}|_{X^{N_i}}$ and 
$T^{f_i}|_{X^{\theta,N_i}}$, are IETs, thus they are not continuous. Therefore we need to discuss in detail the notion of 
unique ergodicity and its relation to uniform convergence of Hopf averages, as well as how the convergence varies as we 
vary the angle on the translation surface. To do this we will use the Cantor representation of the previous subsection.

\begin{lemma}\label{lemlem}
An IET  $(T,X)$ without saddle connection is uniquely ergodic if and only if it's Cantor representation $(\hat T, K)$ is uniquely ergodic
\end{lemma}

\begin{proof} The only if direction is clear.

Consider any pair of  $\hat T$ invariant measures $\hat \nu_1, \hat \nu_2$.
Using the mapping $j$ we pushforward these two measures to a pair of measure  $\nu_j$ on $Y$.
Since $T$ is uniquely ergodic the pushwords  satisfy $ \nu_1 =   \nu_2 $.
Thus for any measurable set $C \subset I$, setting $\hat C := j^{-1}(C)$ we have
\begin{equation}\hat \nu_1(\hat C) = \hat \nu_2(\hat C).\label{haha}\end{equation} 

If $\hat \nu_1 \ne \hat \nu_2$ then there is a measurable set $B \subset K$ such that
$\hat \nu_1(B) \ne \hat \nu_2(B)$.
Consider the sets $\hat{B} := j^{-1}( j(B)) \supset B$ and $B_1 := \hat B \setminus B$. By construction the set $B_1$ is 
at most countable since it is contained in the set consisting of the $\hat T$ orbits  the finite set of $j^{-1}$ preimages of the singular points of $(T,X)$.
By Equation \eqref{haha} we have 
$\nu_1(\hat B) = \nu_2(\hat B)$, thus at least one of the two measures must give positive measure to the countable
set $B_1$, i.e., it must be atomic. 
Since we have supposed that there are no saddle connections,
there can not be any atomic measures, thus $\hat \nu_1 = \hat \nu_2$.
\end{proof}

We need the following result which we did not find in the literature, it is well known if $g \equiv 1$, i.e., for Birkhoff averages.
\begin{proposition}\label{propprop}
If $\hat T : K \to K$ is uniquely ergodic (in the finite measure sense) then for every $f,g \in C(K)$ (with $\int_K g \, d\mu \ne 0$)
$$ \frac{ \sum_{k=0}^{\ell} f \big (\hat T^k( z)  \big)}{\sum_{k=0}^{\ell} g \big (\hat T^k( z)  \big)}$$
converges uniformly to a constant.
\end{proposition}

The proof is essentially identical to the proof in the Birkhoff case.
\begin{proof}
From the Hopf ergodic theorem the constant must be equal 
to $\int_K f  \, d\mu / \int_K g \, d\mu$ where $\mu$ is 
the unique invariant probability measure. 
 From the uniform converges of Birkhoff sums
we can find an $L_0>0$ such that for all $\ell \ge L_0$
and all $z \in K$ we have 
${\sum_{k=0}^{\ell} g \big (\hat T^k( z)  \big)}> 0$ and thus the Hopf sums are well defined.
Suppose that the convergence is not uniform for some $f,g$.
Then there exists $\e > 0$ such that for all $L \ge L_0$
there exists $\ell > L$ and there exists $z_\ell \in K$ with
$$\left |  \frac{ \sum_{k=0}^{\ell} f \big (\hat T^k( z_\ell)  \big)}{\sum_{k=0}^{\ell} g \big (\hat T^k( z_\ell)  \big)} 
-\frac{\int_K f \, d\mu}{ \int_K g \, d\mu} \right | \ge \e.
$$
Set $\mu_\ell := \frac{1}{\ell} \sum_{i=0}^{\ell  -1} \delta_{\hat T^i z_\ell}$, then
$$ \frac{ \sum_{k=0}^{\ell} f \big (\hat T^k( z_\ell)  \big)}{\sum_{k=0}^{\ell} g \big (\hat T^k( z_\ell)  \big)} 
= \frac{\int_K f \, d\mu_\ell}{ \int_K g \, d\mu_\ell}$$
and thus
$$\left | \frac{\int_K f \, d\mu_\ell}{ \int_K g \, d\mu_\ell}
-\frac{\int_K f \, d\mu}{ \int_K g \, d\mu} \right | \ge \e.$$
Choose a convergent subsequence $\mu_{\ell_k} \to \mu_\infty$, then $\mu_\infty$ is an invariant probability measure and
$$\left | \frac{\int_K f \, d\mu_\infty}{ \int_K g \, d\mu_\infty}
-\frac{\int_K f \, d\mu}{ \int_K g \, d\mu} \right | \ge \e.$$
Thus $\mu_\infty \ne \mu$.
\end{proof}

Remark: 
Given a continuous function $h: X\times \SS1 \to \R$ we define $\hat h : K \to \R$ by
setting $\hat h(z) := h(j(z))$. The function $\hat h$ is continuous since $h$ and $j$ are continuous, we will call
it the \textit{lift} of $h$.

\begin{corollary}\label{corcor}
Suppose  $T : X \to X$ is a  uniquely ergodic IET.  Then for every $f,g \in C(X)$ (with $\int_X g \, d\mu \ne 0$)
$$ \frac{ \sum_{k=0}^{\ell} f \big ( T^k( z)  \big)}{\sum_{k=0}^{\ell} g \big (T^k( z)  \big)}$$
converges uniformly to a constant.
\end{corollary}

\begin{proof} We consider the Cantor representation $(\hat T, K)$ of $(T,X)$.  By Lemma \ref{lemlem} it is uniquely ergodic.
We lift $f,g$ to $\hat f,\hat g$ and 
 apply  Proposition  \ref{propprop} to $\hat f$ and $\hat g$.
Projecting back to $X$ yields the uniform convergence of the Hopf averages of $f$ and $g$. 
\end{proof}


\begin{thebibliography}{99} \footnotesize{

\bibitem[AaNaSaSo]{AaNaSaSo} J.\ Aaronson, H.\ Nakada, O.\ Sarig, R.\ Solomyak, 
\textit{Invariant measures and asymptotics for some skew products}
Israel J.\ Math.\ 128 (2002) 93--134.

\bibitem[Ar]{Ar} P.\  Arnoux,  
\textit{Ergodicit\'e g\'en\'erique des billards polygonaux (d'apr\`es Kerckhoff, Masur, Smillie).}
S\'eminaire Bourbaki, Vol.\ 1987/88. Ast\'erisque No.\ 161--162 (1988), Exp. No. 696, 5, 203--221 (1989). 

\bibitem[EhEh]{EhEh} P.\ and T.\ Ehrenfest \textit{Begriffliche Grundlagen 
der statistischen Auffassung in der Mechanik} 
Encykl. d. Math. Wissensch. IV 2 II, Heft 6, 90 S (1912)
(in German, translated in:)
\textit{The conceptual foundations of the
statistical approach in mechanics,} (trans. Moravicsik,
M. J.), 10-13 Cornell University Press, Itacha NY (1959).

\bibitem[FrHu]{FrHu} K.\ Fr\k{a}czek and P.\ Hubert 
\textit{Recurrence and non-ergodicity in generalized wind-tree models}
Math.\ Nachrichen 291  (2018)  1686--711.

\bibitem[FrUl]{FrUl} K.\ Fr\k{a}czek and C.\  Ulcigrai
\textit{Non-ergodic $\mathbb{Z}$-periodic billiards and infinite translation surfaces} Invent.\ Math.\ 197 (2014) 241--298.

\bibitem[FrUl1]{FrUl1} K.\ Fr\k{a}czek and C.\  Ulcigrai \textit{Ergodic directions for billiards in a strip with periodically located obstacles}  Communications in Mathematical Physics 327 (2014) 643--663.

\bibitem[GaKrTr]{GaKrTr} G.\ Galperin, T.\ Kr\"uger and S.\ Troubetzkoy  \textit{Local instability of orbits in polygonal and polyhedral billiards} Comm.\ Math.\ Phys.\ 169 (1995), no. 3, 463--473.
 
 \bibitem[GoLa]{GoLa} S.\ Gou\"ezel and E.\ Lanneau \textit{Un th\'eor\`eme de Kerckhoff, Masur et Smillie: unique ergodicit\'e sur les surfaces plates} (French)  \'Ecole de Th\'eorie Ergodique, 113--145, S\'emin.\ Congr., 20, Soc. Math.\ France, Paris, 2010.

\bibitem[HoHuWe]{HoHuWe} P.\ Hooper, P.\ Hubert and B.\ Weiss \textit{Dynamics on the infinite staircase}
Discrete Cont.\ Dyn.\ Sys.\ 33 (2013) 4341--4347.

\bibitem[Ho]{Ho} P.\ Hooper,  \textit{The invariant measures of some infinite interval exchange maps}
Geom.\ Topol.\ 19 (2015)1895--2038. 

\bibitem[HuLeTr]{HuLeTr} P.\ Hubert, Pascal, S.\ Leli\`evre and S.\ Troubetzkoy \textit{The Ehrenfest wind-tree model: periodic directions, recurrence, diffusion}
J.\ Reine Angew.\ Math. 656 (2011) 223--244. 

\bibitem[HuWe]{HuWe} P.\ Hubert,  and B.\ Weiss  \textit{Ergodicity for infinite periodic translation surfaces}
Composito  Math. 149 (2013) 1364--1380.

\bibitem[Ka]{Ka} A.\ Katok \textit{The growth rate for the number of singular and periodic orbits for a polygonal billiard} Comm.\ 
Math.\ Phys.\ 111 (1987), no. 1, 151--160.

\bibitem[KeMaSm]{KeMaSm} S.\ Kerckhoff, H.\ Masur, and J.\ Smillie, \textit{Ergodicity of billiard flows and quadratic differentials},  Annals of Math.\
(2) 124 (1986), no. 2, 293--311.

\bibitem[MS]{MS} A. M. Málaga Sabogal \textit{Étude d'une famille de transformations préservant la mesure de $\mathbb Z \times T$}
PhD dissertation, Paris-Sud 11 University, 2014

\bibitem[MSTr1]{MSTr1} A.\ M\'alaga Sabogal and S.\ Troubetzkoy
\textit{Minimality of the Ehrenfest wind-tree model}
Journal Modern Dynamics 10 (2016), 209--228. 

\bibitem[MSTr2]{MSTr2} A.\ M\'alaga Sabogal and S.\ Troubetzkoy
\textit{Ergodicity of the Ehrenfest wind-tree model}
Comptes Rendus Mathematique 354 (2016) 1032--1036.

\bibitem[MSTr3]{MSTr4} A.\ M\'alaga Sabogal and S.\ Troubetzkoy
\textit{Infinite ergodic index of the Ehrenfest wind-tree model} 
Comm.\ Math.\ Phys.\ 358 (2018) 995--1006.

\bibitem[MaTa]{MaTa} H.~Masur and S.~Tabachnikov, \textit{Rational billiards and
	flat structures},  Handbook of dynamical systems, Vol.~1A, 1015--1089,
North-Holland, Amsterdam, 2002. 

\bibitem[Mo]{Mo} T.\ Monteil \textit{Introduction to the theorem of Kerkhoff, Masur and Smillie}
Workshop Arbeitsgemeinschaft: Mathematical Billards, Apr 2010, Franckfurt, Germany. pp.955-1015, 2010, Oberwolfach Report


\bibitem[RaRa]{RaRa} K.\ Rafi and A.\  Randecker {www.youtube.com/watch?v=byEo9PpHE8Y}

\bibitem[RaTr]{RaTr} D.\ Ralston and S.\ Troubetzkoy \textit{Ergodic infinite group extensions of geodesic flows on
translation surfaces} Journal Modern Dynamics 6 (2012) 477--497.

\bibitem[Sa]{Sa} O.\ Sarig \textit{Unique ergodicity for infinite measures}
Proceedings of the International Congress of Mathematicians Hyderabad, India, 2010
1777-1803;

\bibitem[Sc]{Sc} K.\ Schmidt \textit{Unique ergodicity and related problems} pp. 188--198 in ``Ergodic theory proceedings Oberwolfach''
Lecture Notes in Mathematics 729, Springer Verlag, Berlin-New York 1979.

\bibitem[Tr]{Tr} S.\ Troubetzkoy \textit{Recurrence in generic staircases}
Discrete and Continuous Dynamical Systems - Series A 32, 3 (2012) 1047-1053

\bibitem[Ve1]{Ve1} W.\ Veech 
\textit{Moduli spaces of quadratic differentials}
J. D'Analyse Math., 55, (1990) 117--171.

\bibitem[Ve2]{Ve2} W.\ Veech 
\textit{Teichmuller curves in moduli space. Eisenstein series and an application to triangular billiards.}
Invent. Math., 97, (1989) 553--583.

}


\end{thebibliography}
\end{document}